\documentclass[11pt]{amsproc}

\usepackage{mathtext}
\usepackage[cp1251]{inputenc}

\usepackage{bm}
\usepackage{xcolor}

\usepackage[dvips]{graphicx}
\usepackage{amsmath}
\usepackage{amssymb}
\usepackage{amsxtra}

\def\N{{{\Bbb N}}}
\def\Z{{{\Bbb Z}}}
\def\T{{{\Bbb T}}}
\def\R{{\Bbb R}}
\def\C{{\Bbb C}}
\def\l{{\lambda }}
\def\a{{\alpha }}
\def\D{{\Delta }}

\def\a{{\alpha}}
\def\b{{\beta}}
\def\d{{\delta}}
\def\e{{\varepsilon}}
\def\s{{\sigma}}
\def\vp{{\varphi}}
\def\t{{\tau }}
\def\g{{\gamma }}
\def\w{{\omega }}

\def\L{{\mathcal{L} }}

\def\){\right)}
\def\({\left(}
\def\supp{\operatorname{supp}}

\numberwithin{equation}{section}

\newtheorem{theorem}{Theorem}[section]

\newtheorem{corollary}[theorem]{Corollary}
\newtheorem{lemma}[theorem]{Lemma}

\newtheorem{proposition}[theorem]{Proposition}
\newtheorem{remark}[theorem]{Remark}
\newtheorem{definition}[theorem]{Definition}
\newtheorem{example}[theorem]{Example}

\par

\begin{document}

%
%
%
%
%

\bigskip

\title[]{
Marcinkiewicz--Zygmund inequalities \\ in quasi-Banach function spaces}

\author[Yurii
Kolomoitsev]{Yurii
Kolomoitsev$^{\text{1}}$}
\address{}
\email{kolomoitsev@gmail.com}

\author[Sergey Tikhonov]{Segey Tikhonov$^\text{a, 2}$}
\address{Centre de Recerca Matem\`atica, Campus de Bellaterra, Edifici C 08193
Bellaterra, Barcelona, Spain; ICREA, Pg. Llu\'is Companys 23, 08010
Barcelona, Spain, and Universitat Aut\'onoma de Barcelona}
\email{stikhonov@crm.cat}


\thanks{$^\text{a}$Centre de Recerca Matem\`atica, Campus de Bellaterra, Edifici C 08193
Bellaterra, Barcelona, Spain; ICREA, Pg. Llu\'is Companys 23, 08010
Barcelona, Spain, and Universitat Aut\'onoma de Barcelona}


\thanks{$^1$Supported by the German Research Foundation in the framework of the RTG 2088}
\thanks{$^2$Supported
by PID2023-150984NB-I00, 2021 SGR 00087, AP 23488596,
the CERCA Programme of the Generalitat de Catalunya, and Severo Ochoa and Mar\'{i}a de Maeztu
Program for Centers and Units of Excellence in R\&D (CEX2020-001084-M)
}


\thanks{E-mail address: kolomoitsev@gmail.com}

\date{\today}

\begin{abstract}
We obtain Marcinkiewicz--Zygmund (MZ) inequalities  in various Banach and quasi-Banach spaces
under minimal assumptions on the  structural properties of these spaces. Our main results show that the Bernstein inequality in a general quasi-Banach function lattice $X$ implies Marcinkiewicz--Zygmund type estimates in $X$. We present a general approach to obtain MZ inequalities not only for polynomials but for other function classes  including  entire functions of exponential type, splines,  exponential sums, etc.
\end{abstract}

\maketitle



\section{Introduction}

The classical Marcinkiewicz--Zygmund (MZ) inequality for trigonometric polynomials in  $L_p(\T)$, $\T=[0,2\pi),$ has the following form
\begin{equation}\label{mzclassical}
C^{-1}\(\frac1{N+1}\sum_{k=0}^{N}|T(t_{k})|^p\)^{1/p}\le \|T\|_{L_p(\T)}\le C\(\frac1{N+1}\sum_{k=0}^{N}|T(t_{k})|^p\)^{1/p},
\end{equation}
where $T$ is a trigonometric polynomial of degree at most $n$,
$\mathcal{T}_n=\{T\,:\, T(x)=\sum_{k=-n}^n c_k e^{ikx}\}$,
$t_k=\frac{2\pi k}{N+1}$, $k=0,\dots,N$,
$$
N=\left\{
                                                \begin{array}{ll}
                                                  2n, & \hbox{if $1<p<\infty$,} \\
                                                  {[(2+\varepsilon)n]+1}, \,\,\varepsilon>0, & \hbox{if $0<p\le 1$ or $p=\infty$,}
                                                \end{array}
                                              \right.
$$
and the constant $C$ depends only on $p$ and $\varepsilon$, cf.~\cite[Ch.~X, \S~7]{Z}, \cite{Pe83}.

This estimate plays a key role in the modern discretization  theory and sampling recovery both in the theoretical and applied aspects. Various generalizations of the MZ inequality are known, see, e.g.,~\cite{CZ99, FHJU24, KKLT22, Kr20, MT00, OP07, RS97, Xu91}.



It is easy to see that \eqref{mzclassical} can be equivalently written as follows
\begin{equation}\label{mzclassical1}
C^{-1}\bigg\|\sum_{k=0}^{N}|T(t_{k})| \chi_{(t_k,t_{k+1})}\bigg\|_{L_p(\T)}\le \|T\|_{L_p(\T)}\le C
\bigg\|\sum_{k=0}^{N}|T(t_{k})| \chi_{(t_k,t_{k+1})}\bigg\|_{L_p(\T)}.
\end{equation}
This formula shows that the $L_p$ norm of a trigonometric polynomial is equivalent to the same norm of a simple function that interpolates the function at given points. This observation gives rise to the following natural question:

\smallskip

\noindent {\it Under what conditions
on the space $X$, the subspace $F_n\subset X$, and the points $(x_k)_{k=0}^N$  does
the  Marcinkiewicz--Zygmund inequality
\begin{equation}\label{form}
   C^{-1}\bigg\|\sum_{k=0}^{N}|F(x_{k})| \chi_{(x_k,t_{x+1})}\bigg\|_{X}\le \|F\|_{X}\le C
\bigg\|\sum_{k=0}^{N}|F(x_{k})| \chi_{(x_k,x_{k+1})}\bigg\|_{X}
\end{equation}
hold for any $F\in F_n$ with a constant $C$ independent of $F_n$ and $N$?
}

\smallskip

\noindent In this paper we answer this question for a Banach and
quasi-Banach  function space $X$, more general than $L_p$ or even having different structure.

{\it Notations.} First, we specify function spaces and (quasi)-norm, for which we generalize the MZ inequalities. Let $X=X(\Omega, \mu; q)$, $0<q\le 1$, be a (quasi-)Banach function lattice on a domain $\Omega\subset \R^d$ endowed with positive $\s$-finite measure~$\mu$
(cf. ~\cite[Ch.~2]{BeSh}). In what follows, we will assume that the functional $\|\cdot\|_X\,:\, X\mapsto \R_+$ satisfies the following properties:

\begin{itemize}
  \item[$1)$] $\|f\|_X=0$ if and only if $f=0$,
  \item[$2)$] $\|\l f\|_X=|\l|\|f\|_X$ for all $\l\in \C$,
  \item[$3)$] $\|f+g\|_X^q\le \|f\|_X^q+\|g\|_X^q$,
  \item[$4)$] if $|f|\le |g|$, then $\|f\|_X\le \|g\|_X$ for all $f,g\in X$,
    \item[$5)$] $\chi_B\in X$ for all $B\subset \Omega$ with $\mu(B)<\infty$.

\end{itemize}

If $X$ is a Banach space (when $q=1$), then we will write $X=X(\Omega, \mu)=X(\Omega, \mu; 1)$ and  assume that $X\subset L_1$. In this case, $X'$ denotes the associate space of $X$ with the norm given by
$$
\|g\|_{X'}=\sup_{\|f\|_X\le 1} \int_\Omega g(x)\overline{f(x)}d\mu(x),
$$
see \cite[p.~12]{BeSh}.

Recall that the space $X$ is called translation invariant if for all $t\in \Omega$ and $f\in X$, there hold $f(x+t)\in X$ and $\|f(\cdot+t)\|_X=\|f\|_X$.
Also, the space $X$ is called rearrangement invariant if $\|f\|_X=\|g\|_X$ whenever
$\mu_f(x)=\mu_g(x)$, where
$\mu_h(x)=\mu(\{x\in \Omega\,:\, |h(x)|\ge y\})$
is the distribution function of $h$.

In what follows, $\mathcal{F}_n$ denotes a subspace of $X\cap C(\R^d)$, which depends on the parameter $n$, where
$$
n\in \mathbb{N}\quad\text{or}\quad n\,\,\text{is a real positive number}.
$$
We note that for all classical examples of subspaces $\mathcal{F}_n$ (trigonometric and algebraic polynomials, splines, entire functions of exponential type) there holds the condition $\mathcal{F}_m\subset \mathcal{F}_{n}$ for $n\ge m$. However, we do not require this assumption in our main results.



Throughout the paper, we will use the letter
$C$ to denote various positive constants, which may differ even within the same line where they appear.
We also use the notation
$
\, A \lesssim B,
$
with $A,B\ge 0$, for the
estimate
$\, A \le C\, B,$ where $\, C$ is a positive constant independent of
the essential variables in $\, A$ and $\, B$ (usually, $\d$, $n$, and $N$). 
 If $\, A \lesssim B$
and $\, B \lesssim A$ simultaneously, we write $\, A \asymp B$ and say that $\, A$
is equivalent to $\, B$.
As usual, for $\a\in \Z_+^d$,
$$
D^\a=\frac{\partial^{|\a|_1}}{\partial x_1^{\a_1}\dots\partial x_d^{\a_d}}.
$$
For $\d>0$, we denote
 $$
 Q(x,\d)=\{y\in \Omega\,:\,\|x-y\|_\infty\le \d\}.
 $$

{\it Structure of the paper.} In Section~2 we establish MZ inequalities in Banach and quasi-Banach spaces: in Subsection~2.1, we investigate the boundedness of certain maximal and minimal functions in $X$ and then in Subsection~2.2 we apply these results  to derive MZ inequalities in form~\eqref{form}.
In Section~3,
we obtain MZ-type inequalities
with minimal number of nodes with the help of  quadrature formulas.
Sections~4 and~5 are devoted to  MZ estimates for special and important  examples, namely,  for trigonometric and algebraic polynomials.
Examples for other function systems are considered in Section~6.
In Section~7, we consider some applications of MZ-type inequalities including the approximation of functions by sampling operators and polynomials in general Banach spaces, Nikolskii-type inequalities, and embedding theorems.








\section{MZ-type inequalities in Banach and quasi-Banach spaces}

\subsection{Boundedness of maximal and minimal functions  
}


We start with boundedness results for certain  maximal and minimal functions of elements from $\mathcal{F}_n$ in transaltion invariant function lattices $X$.

\begin{theorem}\label{thmm2}
Let $X$ be a translation invariant Banach lattice. Assume that there exist $\b>0$ and $B>0$ such that
   \begin{equation}\label{ber}
         \|D^\a F_n\|_X\le (Bn^\beta)^{|\a|_1}\| F_n\|_X,\quad \a\in\{0,1\}^d,\quad F_n\in \mathcal{F}_n.
   \end{equation}
Then the following assertions hold:

\medskip

  \noindent{\rm (A)} For any $F_n\in \mathcal{F}_n$ and any $A>0$ satisfying
  \begin{equation}\label{CCC}
    \eta=(1+2AB)^d-1<1,
  \end{equation}
  we have
   \begin{equation}\label{mmn}
      \frac{1}{1+\eta}\bigg\|\max_{t\in Q(x,\frac{A}{n^\b})}|F_n(t)|\bigg\|_X \le \|F_n\|_X\le \frac{1}{1-\eta}\bigg\|\min_{t\in Q(x,\frac{A}{n^\b})}|F_n(t)|\bigg\|_X.
 \end{equation}

 \noindent{\rm (B)} The left-hand side inequality in~\eqref{mmn} holds for any $A>0$, that is, we do not require that $\eta<1$.
\end{theorem}

\begin{proof}
By the generalized Newton-Leibnitz formula (see, e.g.,~\cite[Ch.~4, \S~16]{BIN75}), we have
  \begin{equation}\label{nl}
    F_n(x+t)-F_n(x)=\sum_{
    \a\in \{0,1\}^d \setminus  \{\bar{0}\}}\int_{[0,t]^\a} D^\a F_n(x+\a v)(dv)^\a,
  \end{equation}
  where $\a v=(\a_1v_1,\dots,\a_dv_d)$. Denote $\d:=\tfrac{A}{n^\b}$. Applying~\eqref{nl}, the triangle inequality, Lemma~\ref{lecon}, and Bernstein-type inequality~\eqref{ber}, we obtain
  \begin{equation}\label{mm1-}
    \begin{split}
       \bigg\|\max_{t\in Q(0,\d)}|F_n(x+t)-F_n(x)|\bigg\|_X&\le \sum_{\a\in \{0,1\}^d \setminus  \{\bar{0}\}}\bigg\|\int_{[-\d,\d]^\a} |D^\a F_n(x+\a v)|(dv)^\a\bigg\|_X\\
       &\le \sum_{\a\in \{0,1\}^d \setminus  \{\bar{0}\}} (2\d)^{|\a|_1}\|D^\a F_n\|_X\\
       &\le \sum_{\a\in \{0,1\}^d \setminus  \{\bar{0}\}} (2\d)^{|\a|_1}(Bn^\beta)^{|\a|_1}\|F_n\|_X\\
       &= \sum_{\a\in \{0,1\}^d \setminus  \{\bar{0}\}} (2AB)^{|\a|_1}\|F_n\|_X\\
       &=\((1+2AB)^d-1\)\|F_n\|_X=\eta\|F_n\|_X.
     \end{split}
  \end{equation}
 Using again the triangle inequality and~\eqref{mm1-}, we have
  \begin{equation*}
    \begin{split}
        \bigg\|\max_{{t\in Q(0,\d)}}|F_n(x+t)|\bigg\|_X-\|F_n\|_X&\le \bigg\|\max_{{t\in Q(0,\d)}}|F_n(x+t)|-|F_n(x)|\bigg\|_X\\
        &\le \bigg\|\max_{{t\in Q(0,\d)}}|F_n(x+t)-F_n(x)|\bigg\|_X\le \eta\|F_n\|_X,
     \end{split}
  \end{equation*}
which implies the left-hand side inequality in~\eqref{mmn}. Similarly, 
 we get
\begin{equation*}
  \begin{split}
      \|F_n\|_X-\bigg\|\min_{{t\in Q(0,\d)}}|F_n(x+t)|\bigg\|_X&\le \bigg\| |F_n(x)|-\min_{{t\in Q(0,\d)}}|F_n(x+t)|\bigg\|_X\\
      &\le
      \bigg\|\max_{{t\in Q(0,\d)}}|F_n(x)-F_n(x+t)|\bigg\|_X\le \eta\|F_n\|_X.
   \end{split}
\end{equation*}
This gives the right-hand side inequality in~\eqref{mmn}.
\end{proof}

The following remarks are in order.

\begin{remark}\label{rem1}

%
%
We note that the size of the cube $Q(0,\tfrac{A}{n^\b})$  cannot be enlarged. 
  Take, for example, $X=L_p(\T)$, $\b=1$,  and $F_n(x)=\sin nx$. Then
  $\min\limits_{{t\in Q(x,\d)}}|F_n(t)|=0$ for all $x\in \T$ if $\d=\frac{C}{n}$ for sufficiently large $c$.
\end{remark}

\begin{remark}\label{remark2}
$(i)$ Carefully checking the proof of Theorem~\ref{ber}, we observed that the conclusion of this theorem remains valid for any Banach function lattice $X=X(\Omega,\mu)$ if we assume that

  \begin{itemize}
    \item[$1)$]  there exist $\b>0$ and $B>0$ such that
   \begin{equation*}
      D^\a F_n\in \mathcal{F}_n\quad\text{and}\quad   \|D^\a F_n\|_X\le (Bn^\beta)^{|\a|_1}\| F_n\|_X,\quad \a\in\{0,1\}^d,\quad F_n\in \mathcal{F}_n;
   \end{equation*}

    \item[$2)$] for any $A>0$ there exists a constat $c=c(A,X)$ such that
    \begin{equation*}
    \|F_n(\cdot+t)\|_X\le c\|F_n\|_X,\quad t\in Q(0,\tfrac{A}{n^\b}),\quad F_n\in \mathcal{F}_n.
  \end{equation*}
  \end{itemize}
In this case, the corresponding constant $\eta$ and condition~\eqref{CCC} should be replaced by $\eta=c((1+2AB)^d-1)<1$.


$(ii)$  Repeating the proof of Theorem~\ref{thmmwe} below, one can show that condition 2)  holds with the constant $c=e^{dAB}-1$. 
  %
%
\end{remark}

\begin{remark}\label{remm2.5}
  Let $X=X(\Omega,\mu)$ be a Banach function lattice.  Assume that $\mathcal{F}_n$ is an $n$-dimensional subspace of $C(\R^d)$ such that for any $F_n \in \mathcal{F}_n$ and $t\in \Omega$ we have $F_n(\cdot+t)\in \mathcal{F}_n$ and
$\|F_n(\cdot+t)\|_X\le c\|F_n\|_X$, where $c$ depends only on $X$. If
$$
\|D^\a F_n\|_{L_\infty(\Omega)}\le (Bn^\beta)^{|\a|_1}\| F_n\|_{L_\infty(\Omega)},\quad \a\in\{0,1\}^d,\quad F_n\in \mathcal{F}_n,
$$
then~\eqref{ber} holds in the following form
   \begin{equation*}
         \|D^\a F_n\|_X\le c(Bn^\beta)^{|\a|_1}\| F_n\|_X,\quad \a\in\{0,1\}^d,\quad F_n\in \mathcal{F}_n.
   \end{equation*}
This follows from Lemma~\ref{ilf}, see also the proof of Proposition~\ref{pMarkX}.
\end{remark}

\begin{remark}
Under the conditions of Theorem~\ref{thmm2}, for any $F_n\in \mathcal{F}_n$ and  $0<\d\le\frac{A}{n^\b}$ with $A$ satisfying~\eqref{CCC}, the following relation holds
  $$
  \|F_n\|_X\asymp \bigg\|\frac1{\mu(Q_\d(x,\d))}\int_{Q(x,\d)}|F_n(x)|d\mu(x)\bigg\|_X.
  $$
\end{remark}


Our next result provides an extension
of Theorem~\ref{thmm2} in several directions.
First, we derive an analogue of~\eqref{mmn}
for quasi-Banach spaces $X$, not necessarily  translation invariant.
Second, in place of the classical Bernstein type inequality giving by $\|F_n'\|_X\lesssim n^\beta\|F_n\|_X$,
we assume its weighted version:
$\|\vp_n F_n'\|_X\lesssim n^\beta\|F_n\|_X$. The weight $\vp_n$  plays the crucial role for non-periodic case when $\Omega$ is a bounded domain.
For example, for algebraic polynomial $P_n$ and $X=L_p[-1,1]$,  the Markov inequality states that $\|P_n'\|_X\lesssim n^2\|P_n\|_X$ but  the Bernstein inequality
$\|\vp_n P_n'\|_X\lesssim n\|P_n\|_X$ with
$\vp_n(x)=\sqrt{1-x^2}+\tfrac1n$
provides
a more accurate bound, cf.
 Remark~\ref{exPnS}.

On the other hand, unlike
Theorem ~\ref{thmm2}, in the next result we require
 that the elements of $\mathcal{F}_n$ have infinitely many derivatives. This is not the case for spline functions,  see Section
\ref{splines-}.



\begin{theorem}\label{thmmwe}
Let $X=X(\Omega, \mu; q)$, $0<q\le 1$, be a  quasi-Banach lattice.
  Assume that there exist $\b>0$ and a nonnegative function $\vp_n$ on $\Omega$  such that for all $\a\in\Z_+^d$
  \begin{equation}\label{berwe}
    \|\vp_n^{|\a|_1}D^\a F_n\|_X\le B_\a n^{\beta|\a|_1}\| F_n\|_X,\quad F_n\in \mathcal{F}_n,
  \end{equation}
 with $B_\a>0$. Then the following assertions hold:

\smallskip

\noindent{\rm (A)} For each $F_n\in \mathcal{F}_n$ and any constant $A>0$ satisfying
\begin{equation}\label{asumpeta}
  \eta=\sum_{\a\in\Z_+^d\setminus  \{\bar{0}\}}\(\frac{A^{|\a|_1}B_\a}{\a!}\)^{q}<1,
\end{equation}
 we have
  \begin{equation}\label{eqwe1}
  \begin{split}
           \frac1{\(1+\eta\)^{1/q}}\bigg\|\max_{t\in Q\(x, \tfrac{A \vp_n(x)}{n^\beta}\)}|F_n(t)|\bigg\|_X
           &\le \|F_n\|_X\\
           &\le \frac1{\(1-\eta\)^{1/q}}\bigg\|\min_{t\in Q\(x, \tfrac{A \vp_n(x)}{n^\beta}\)}|F_n(t)|\bigg\|_X.
  \end{split}
  \end{equation}

 \noindent{\rm (B)} The left-hand side inequality in~\eqref{eqwe1} holds for any $A>0$, that is, we do not require that $\eta<1$.
\end{theorem}

\begin{proof}
Denote $\d:=\frac{A}{n^\beta}$.  Applying  the Taylor formula,  the triangle inequality, and Bernstein inequality~\eqref{berwe}, we obtain
  \begin{equation}\label{mm1we}
    \begin{split}
       \bigg\|\max_{t\in Q(0,\d)}&|F_n(x+\vp_n(x)t)-F_n(x)|\bigg\|_X^q\\
       &=\bigg\|\max_{t\in Q(0,\d)}\bigg|\sum_{\a\in\Z_+^d\setminus  \{\bar{0}\}} \frac{D^\a F_n(x)}{\a!}(\vp_n(x)t)^\a\bigg|\bigg\|_X^q\\
    &\le\bigg\|\max_{t\in Q(0,\d)}\sum_{\a\in\Z_+^d\setminus  \{\bar{0}\}} \frac{|D^\a F_n(x)|}{\a!}|t_1|^{\a_1}\dots|t_d|^{\a_d}(\vp_n(x))^{|\a|_1}\bigg\|_X^q\\
        &\le \sum_{\a\in\Z_+^d\setminus  \{\bar{0}\}}\frac{\|(\vp_n(x))^{|\a|_1} D^\a F_n\|_X^q}{(\a!)^q}\d^{|\a|_1q}\\
        &\le \sum_{\a\in\Z_+^d\setminus  \{\bar{0}\}}\frac{((\d n^\b)^{|\a|_1} B_\a)^{q}}{(\a!)^q}\|F_n\|_X^q\\
        &= \sum_{\a\in\Z_+^d\setminus  \{\bar{0}\}}\(\frac{A^{|\a|_1}B_\a}{\a!}\)^q\|F_n\|_X^q=\eta\|F_n\|_X^q.
     \end{split}
  \end{equation}
 Using again the triangle inequality and~\eqref{mm1we}, we have
  \begin{equation*}
    \begin{split}
        \bigg\|\max_{t\in Q(0,\d)}|F_n(x+t\vp_n(x))|\bigg\|_X^q-\|F_n\|_X^q&\le \bigg\|\max_{t\in Q(0,\d)}|F_n(x+t\vp_n(x))|-|F_n(x)|\bigg\|_X^q\\
        &\le \bigg\|\max_{t\in Q(0,\d)}|F_n(x+t\vp_n(x))-F_n(x)|\bigg\|_X^q\\
        &\le \eta\|F_n\|_X^q,
     \end{split}
  \end{equation*}
which implies the left-hand side inequality in~\eqref{eqwe1}.
Similarly, we derive  the right-hand side inequality in~\eqref{eqwe1}, 
completing the proof. 
\end{proof}

\begin{remark}
 ${(i)}$ Assuming that $\frac{\partial}{\partial x_i} F_n \in \mathcal{F}_n$ for all $F_n \in \mathcal{F}_n$ and $i=1,\dots,d$, the Bernstein inequality given by \eqref{berwe} can be replaced by
   \begin{equation*}
    \bigg\|\vp_n^\ell\frac{\partial}{\partial x_i} F_n\bigg\|_X\le B_\ell' n^\beta\|\vp_n^{\ell-1} F_n\|_X,\quad i=1,\dots,d,\quad \ell\in \N,\quad F_n\in \mathcal{F}_n.
  \end{equation*}
  In this case, the assumption~\eqref{asumpeta} should be replaced by
  \begin{equation*}
  \eta'=\sum_{\a\in\Z_+^d\setminus  \{\bar{0}\}}\(\frac{A^{|\a|_1}}{\a!} \prod_{i=1}^d\prod_{j=1}^{\a_i}B_j'\)^q<1.
\end{equation*}

\noindent ${(ii)}$ If the Bernstein inequality has the following form
  \begin{equation*}
    \|\vp_n^{|\a|_1}D^\a F_n\|_X\le B_\a (\gamma(n))^{|\a|_1}\| F_n\|_X,\quad F_n\in \mathcal{F}_n,
  \end{equation*}
 for some function $\g$ (cf.~\eqref{berwe}), then in the assertion of Theorem~\ref{thmmwe}
 the constant $\frac{A}{n^\b}$ should be replaced by
 $\frac{A}{\g(n)}$. For example, in the case $\mathcal{F}_n=\mathcal{P}_n^d$ is the set of all $d$-variate algebraic polynomials of total degree $n$ and $X=L_p(S)$, where $0< p\le \infty$ and $S$ is a bounded convex domain, we have the Markov inequality, in which  $\gamma(n)=n^2$, see Remark~\ref{exPnS} below and also Subsection~\ref{secME} for special choices of $\g$.
\end{remark}

%

%

\begin{remark}\label{exPnS}
 Consider the case, where $\mathcal{F}_n=\mathcal{P}_n^d$. Ditzian~\cite{D92} established the following Bernstein-Markov type inequalities: for a bounded convex set $S\subset \R^d$, any direction $\xi\in \R^d$ $(|\xi|=1)$, $r\in \N$, and $0<p\le \infty$,
 \begin{equation*}
   \bigg\|\bigg(\frac{\partial}{\partial\xi}\bigg)^r P_n\bigg\|_{L_p(S)}\le C(p,r,S)n^{2r}\|P_n\|_{L_p(S)}
 \end{equation*}
 and
 \begin{equation*}
   \bigg\| \widetilde{\rho}(\cdot,\xi)^{r/2}\bigg(\frac{\partial}{\partial\xi}\bigg)^r P_n\bigg\|_{L_p(S)}\le C(p,r)n^r\|P_n\|_{L_p(S)},
 \end{equation*}
 where 
 $$
 \widetilde{\rho}(v,\xi):=\sup_{\l>0,\, v+\l\xi\in S}\rho(v,v+\lambda \xi)\sup_{\l<0,\, v+\l\xi\in S}\rho(v,v+\lambda \xi)
 $$
and $\rho$ is the Euclidean distance.  Thus, Theorem~\ref{thmmwe} can be applied in the case, where $\Omega$ is any convex domain and $\vp_n(x)\lesssim\widetilde{\rho}(x,e_i)^{1/2}+\tfrac1n$, $x\in S$, for each $i=1,\dots,d$. In particular, if $\Omega=[-1,1]$, we can take $\vp_n(x)=\sqrt{1-x^2}+\tfrac1n$. See~\cite{DP22} for further results in this direction.
\end{remark}

\subsection{General MZ-type inequalities in quasi-Banach spaces}


Note that in the results given in this section, $N$ can take infinite values. In particular, this is the case of entire functions of exponential type and $\Omega=\R^d$, see  Subsection \ref{efet}.

First, we consider translation invariant lattices.

\begin{theorem}\label{corMZinv}
Let $X$ be a translation invariant Banach lattice and  there exist $\b>0$ and $B>0$ such that
   \begin{equation*}
         \|D^\a F_n\|_X\le (Bn^\beta)^{|\a|_1}\| F_n\|_X,\quad \a\in\{0,1\}^d,\quad F_n\in \mathcal{F}_n.
   \end{equation*}
Assume that there exist  sets $\Omega_k$, $k=1,\dots,N$, where $N=N(n)$, such that
  \begin{equation}\label{mzinv0}
  1\le \sum_{k=1}^{N} \chi_{\Omega_k}(x)\le c_d\quad\text{for any}\quad x\in \Omega
  \end{equation}
  and, for any constant $A>0$ satisfying $\eta=(1+2AB)^d-1<1$ and each $k=1,2,\dots,N$,
  \begin{equation}\label{mzinv0+}
  \Omega_k\subset Q(x,\tfrac{A}{n^\b})\quad\text{for any} \quad x\in \Omega_k.
  \end{equation}
Then the following assertions hold:

\medskip

 \noindent {\rm (A)} For any $F_n\in \mathcal{F}_n$, we have 
  \begin{equation}\label{mzinv1}
    \frac{1}{c_d(1+\eta)}\bigg\|\sum_{k=1}^{N} \max_{t\in\Omega_k}|F_n(t)|\chi_{\Omega_k}\bigg\|_X\le \|F_n\|_X\le \frac{1}{1-\eta}\bigg\|\sum_{k=1}^{N} \min_{t\in\Omega_k}|F_n(t)|\chi_{\Omega_k}\bigg\|_X.
  \end{equation}

\noindent {\rm (B)} The left-hand side inequality in~\eqref{mzinv1} is valid  if we only assume that  the right-hand side of~\eqref{mzinv0} holds and~\eqref{mzinv0+} holds for any\footnote{That is, we do not require that $\eta<1$.} positive $A$.

\end{theorem}

\begin{proof}
 Denote $\d:=\frac{A}{n^\b}$. Applying the right-hand side inequality in Theorem~\ref{thmm2}, we have
  \begin{equation*}
    \begin{split}
       (1-\eta)\|F_n\|_X&\le \bigg\|\min_{{t\in Q(x,\d)}}|F_n(t)|\bigg\|_X\le \bigg\|\sum_{k=1}^N \min_{{t\in Q(x,\d)}}|F_n(t)| \chi_{\Omega_k}\bigg\|_X\\
        &\le \bigg\|\sum_{k=1}^N \min_{t\in \Omega_k}|F_n(t)|\chi_{\Omega_k}\bigg\|_X.
     \end{split}
  \end{equation*}
 Similarly, using the left-hand side inequality in Theorem~\ref{thmm2}, we obtain
   \begin{equation*}
    \begin{split}
       (1+\eta)\|F_n\|_X&\ge \bigg\|\max_{{t\in Q(x,\d)}}|F_n(t)|\bigg\|_X\ge \frac1{c_d}\bigg\|\sum_{k=1}^N \max_{{t\in Q(x,\d)}}|F_n(t)|\chi_{\Omega_k}\bigg\|_X\\
        &\ge \frac1{c_d}\bigg\|\sum_{k=1}^N \max_{t\in \Omega_k}|F_n(t)|\chi_{\Omega_k}\bigg\|_X.
     \end{split}
  \end{equation*}
\end{proof}

Second, we present MZ inequalities just assuming that $X$ is a quasi-Banach lattice. The proof is similar to the one  for Theorem~\ref{corMZinv}.

\begin{theorem}\label{corthmmwe}
Let $X=X(\mu, \Omega; q)$, $0<q\le 1$, be a  quasi-Banach function lattice and
there exist $\b>0$ and a nonnegative function $\vp_n$ on $\Omega$  such that for all $\a\in\Z_+^d$
  \begin{equation*}
    \|\vp_n^{|\a|_1}D^\a F_n\|_X\le B_\a n^{\beta|\a|_1}\| F_n\|_X,\quad F_n\in \mathcal{F}_n,
  \end{equation*}
 with some $B_\a>0$. Assume that
there exist sets $\Omega_k$ and $\Omega_k^*$, $k=1,\dots,N$, where $N=N(n)$, such that
   \begin{equation}\label{mzwe0}
  1\le \sum_{k=1}^N \chi_{\Omega_k}(x)\le c_d\quad\text{for any}\quad x\in \Omega
  \end{equation}
  and, for any constant $A>0$ satisfying $\eta=\sum_{\a\in\Z_+^d,\, \a\neq 0}\(\frac{A^{|\a|_1}B_\a}{\a!}\)^{q}<1$ and for each $k=1,2,\dots,N$,
  \begin{equation}\label{zvv}
   \Omega_k^*\subset Q\(x,\tfrac{A}{n^\b}\vp_n(x)\)\quad\text{for any}\quad x\in \Omega_k.
  \end{equation}
Then the following assertions  hold:

\smallskip

  \noindent {\rm (A)} For any $F_n\in \mathcal{F}_n$, we have
  \begin{equation}\label{mzwe}\quad
    \frac{1}{c_d\(1+\eta\)^{1/q}} \bigg\|\sum_{k=1}^N \max_{y\in\Omega_k^*}|F_n(y)|\chi_{\Omega_k}\bigg\|_X\le \|F_n\|_X\le \frac{1}{\(1-\eta\)^{1/q}} \bigg\|\sum_{k=1}^N \min_{y\in\Omega_k^*}|F_n(y)|\chi_{\Omega_k}\bigg\|_X.
  \end{equation}

 \noindent {\rm (B)} The left-hand side inequality in~\eqref{mzwe} is valid  if we only assume that  the right-hand side of~\eqref{mzwe0} holds and~\eqref{zvv} holds for any positive $A$.

\end{theorem}


\begin{remark}
\begin{itemize}
    \item[$(i)$]  The presence  of the weight function in the Bernstein inequality is essential  for deriving the optimal MZ inequalities, see, for example,  Section~\ref{seca}.

    \item[$(ii)$]
     In the case $\vp_n\equiv 1$, we can take $\Omega_k^*=\Omega_k$. Then Theorem~\ref{corthmmwe} holds if condition~\eqref{zvv} is replaced by 
  $$
  {\rm diam}\, \Omega_k\le \tfrac{A}{n^\b}.
  $$

\end{itemize}

\end{remark}

%

\section{MZ-type inequalities with minimal number of nodes}\label{secOpt}

Let $\Omega$ be a non-empty domain in $\R^d$ and let $\mu$ be a positive $\s$-finite measure on $\Omega$.
Consider the inner product
\begin{equation*}
  (f,g)=
  \int_\Omega f(x)\overline{g(x)}d\mu(x),\quad f,g\in L_{2}(\Omega,\mu).
\end{equation*}
%
Let $(\psi_k)_{k=0}^\infty$ with $\psi_0=1$ be an orthonormal system of functions with respect to this inner product, i.e.,
$$
(\psi_k,\psi_j)=\left\{
                \begin{array}{ll}
                  1, & \hbox{$k=j$,} \\
                  0, & \hbox{$k\neq j$.}
                \end{array}
              \right.
$$
The Fourier series of $f$ and the corresponding $n$th Fourier partial sum are defined by
$$
f\sim \sum_{k=0}^\infty (f,\psi_k)\psi_k\quad\text{and}\quad S_n f=\sum_{k=0}^n(f,\psi_k)\psi_k,
$$
respectively.
We denote by $\Psi_n$ a finite-dimensional subset of $X$ spanned by $(\psi_k)_{k=0}^n$.
(In this section, we assume that $n\in \N$.)


\subsection{MZ-type inequalities under quadrature formulas and Bernstein's inequalities}

The main result in this section is the following general theorem, which in particular cases of trigonometric and algebraic polynomials provides MZ-type inequalities with {\it minimal} number of nodes, see Theorems~\ref{thF} and~\ref{algw} below.

\begin{theorem}\label{thOr2}
  Let $X=X(\Omega,\mu)$ be a Banach lattice.
  Assume that, for a given $n\in \N$, the following conditions hold:

  \begin{enumerate}

    \item[$1)$]  there exist $(x_k)_{k=1}^N$ and $(\mu_k)_{k=1}^N$, $\mu_k\ge 0$, such that
     \begin{equation*}
        \int_\Omega P_{n}(x)Q_n(x)d\mu(x)=\sum_{k=1}^N P_{n}(x_k)Q_n(x_k)\mu_k\quad\text{for all}\quad P_{n},Q_n\in \Psi_n;
     \end{equation*}

    \item[$2)$] there exist $(\Omega_k)_{k=1}^N\subset \Omega$ and absolute constants $c_1$ and $c_2$ such that
    $$
    \mu_k\le c_1\int_{\Omega_k}d\mu(x),\quad k=1,\dots,N,\quad{and}\quad
     \sum_{k=1}^N\chi_{\Omega_k}(x)\le c_2\quad\text{for all}\quad x\in \Omega;
     $$

    \item[$3)$] there exists $C_1>0$ such that
    $$
    \|S_n g\|_{X}\le C_1\|g\|_{X},\quad g\in X;
    $$
    \item[$4)$] 
  there exist $\b>0$ and a non-negative function $\vp_n$ on $\Omega$  such that for all $\a\in\Z_+^d$
  \begin{equation*}
    \|\vp_n^{|\a|_1}D^\a P_n\|_Y\le B_\a n^{\beta|\a|_1}\| P_n\|_Y,\quad P_n\in \Psi_n,\quad Y\in \{X,X'\},
  \end{equation*}
 with $B_\a>0;$

   \item[$5)$] there exists $A>0$ such that $\eta=\sum_{\a\in\Z_+^d\setminus\{0\}}\frac{A^{|\a|_1}B_\a}{\a!}<\infty$ and, for each $k=1,\dots,N$,
   \begin{equation*}
     x_k\in Q\(x,\tfrac{A}{n^\b}\vp_n(x)\)\quad\text{for all}\quad x\in \Omega_k.
   \end{equation*}
    \end{enumerate}
Then, for all $P_{n}\in \Psi_n$, we have
         \begin{equation*}
      \frac{1}{c_2(1+\eta)}\bigg\|\sum_{k=1}^N |P_n(x_k)|\chi_{\Omega_k}\bigg\|_{X}\le\|P_n\|_{X}\le{c_1 c_2C_1(1+\eta)}\bigg\|\sum_{k=1}^N |P_n(x_k)|\chi_{\Omega_k}\bigg\|_{X}.
     \end{equation*}
\end{theorem}

In many cases, the above conditions can be simplified, see, e.g., Theorem~\ref{algw}, in which $\Psi_n=\mathcal{P}_n$ is the set of all algebraic polynomials of degree at most $n$.

To prove Theorem~\ref{thOr2}, we first establish the following result.

\begin{proposition}\label{prmz}
Assume that under the conditions 1), 2), and 3) of Theorem~\ref{thOr2}, there exists a constant $C_2$ independent of $n$ such that
%
%
%
      \begin{equation}\label{pre2}
        \bigg\|\sum_{k=1}^N |P_n(x_k)|\chi_{\Omega_k}\bigg\|_{X'}\le C_2\|P_n\|_{X'},\quad P_n\in \Psi_{n}.
     \end{equation}
     Then
    \begin{equation*}
        \|P_n\|_{X}\le c_1C_1C_2\bigg\|\sum_{k=1}^N |P_n(x_k)|\chi_{\Omega_k}\bigg\|_{X},\quad P_n\in \Psi_{n}.
     \end{equation*}
\end{proposition}

\begin{proof}
  By duality, orthogonality, conditions 1) and 2), and H\"older inequality,  we obtain
  \begin{equation}\label{prpr1}
    \begin{split}
       \|P_n\|_{X}&=\sup_{\|g\|_{X'}\le 1}\int_\Omega P_n(x)\overline{g(x)}d\mu(x)=\sup_{\|g\|_{X'}\le 1}\int_\Omega P_n(x)\overline{S_n g(x)}d\mu(x)\\
       &=\sup_{\|g\|_{X'}\le 1}\sum_{k=1}^N P_n(x_k)\overline{S_n g(x_k)}\mu_k\le\sup_{\|g\|_{X'}\le 1}\sum_{k=1}^N |P_n(x_k){S_n g(x_k)}|\mu_k\\
       &\le c_1\sup_{\|g\|_{X'}\le 1}\int_{\Omega}\(\sum_{k=1}^N |P_n(x_k)|\chi_{\Omega_k}(x)\)\(\sum_{\ell=1}^N |S_n g(x_\ell)|\chi_{\Omega_\ell}(x)\)d\mu(x)\\
       &\le c_1\sup_{\|g\|_{X'}\le 1} \bigg\|\sum_{k=1}^N |P_n(x_k)|\chi_{\Omega_k}\bigg\|_{X}\bigg\|\sum_{\ell=1}^N |S_n g(x_\ell)|\chi_{\Omega_\ell}\bigg\|_{X'}.
     \end{split}
  \end{equation}
  Next, applying \eqref{pre2} and condition 3), we have
  \begin{equation}\label{prpr2}
    \begin{split}
       \bigg\|\sum_{\ell=1}^N |S_n g(x_\ell)|\chi_{\Omega_\ell}\bigg\|_{X'}\le C_2\|S_n g\|_{X'}\le C_1C_2\|g\|_{X'}.
     \end{split}
  \end{equation}
  Finally, combining~\eqref{prpr1} and~\eqref{prpr2}, we complete the proof.
\end{proof}



\begin{proof}[Proof of Theorem~\ref{thOr2}]
  According to Proposition~\ref{prmz}, it is enough to verify that
       \begin{equation}\label{pre2WW}
        \bigg\|\sum_{k=1}^N |P_n(x_k)|\chi_{\Omega_k}\bigg\|_{X'}\le C_2\|P_n\|_{X'}\quad\text{for all}\quad P_{n}\in \Psi_n.
     \end{equation}
 Applying Theorem~\ref{corthmmwe}(B) for the space $X'$ and  $\Omega_k^*=\{x_k\}$, we see that~\eqref{pre2WW} indeed holds.
\end{proof}


\subsection{Approach to MZ inequalities using generalized translation}

In the case when $X$ is an Orlicz space, we can apply a different approach to establish~\eqref{pre2}.
For the definitions of the Orlicz space $L_\Phi$ and the corresponding norm $\|\cdot\|_\Phi$, see Appendix~\ref{Afs}.


We define the generalized translation operator $T^s$, $s\in \Omega$, by
$$
T^s f(x)\sim \sum_{k=1}^\infty (f,\psi_k)\psi_k(x)\psi_k(s)
$$
and assume that there exists a constant $C>0$ such that, for all $s\in \Omega$ and $f\in L_{1}(\Omega,\mu)$,
\begin{equation}\label{tr}
  \| T^s f\|_{L_{1}(\Omega,\mu)}\le C\|f\|_{L_{1}(\Omega,\mu)}.
\end{equation}
Note that in most cases
$$
T^s f(x)=\int_\Omega f(t)K(x,s,t)d\mu(t),
$$
where $K$ is a kernel, which possesses the following properties:
\begin{enumerate}
    \item[$1)$] $K(x,s,t)\ge 0$;
  \item[$2)$] $\int_\Omega K(x,s,t)d\mu(t)=1$;
  \item[$3)$] $K(x,s,t)=K(s,x,t)=K(t,s,x)$.
\end{enumerate}
We will call $T^s$ admissible if the properties 1)-3) hold. In this case \eqref{tr} is clearly valid.

\begin{theorem}\label{thOr}
Let $L_{\Phi}=L_{\Phi}(\Omega,\mu)$ be an Orlicz space. Assume that
conditions 1)--3) of Theorem~\ref{thOr2} hold with $X=L_{\Phi}$ and
 $(\Omega_k)_{k=1}^N$ such that $\Omega_k\cap \Omega_j=\emptyset$, $k\neq j$. Assume, additionally, that
there exists a de la Vall\'ee Poussin type kernel $V_n\in \Psi_{2n}$ such that $(V_n,\psi_k)=1$, $k=0,\dots,n$,
   \begin{equation}\label{vp1}
     \sup_n \|V_{n}\|_{L_{1}(\Omega,\mu)}<\infty,
   \end{equation}
   and
\begin{equation}\label{VVV}
  \sup_{t\in \Omega}\bigg\| \sum_{k=1}^N |T^tV_n(x_k)|\chi_{\Omega_k}\bigg\|_{L_1(\Omega,\mu)}\le C_2,
\end{equation}
where the constant $C_2$ is independent of $n$ and $N$.
%
Then, for all $P_{n}\in \Psi_n$, we have
     \begin{equation*}
           \|P_n\|_{{\Phi}}\asymp\bigg\|\sum_{k=1}^N |P_n(x_k)|\chi_{\Omega_k}\bigg\|_{{\Phi}}.
     \end{equation*}
%
\end{theorem}

\begin{proof}
  The proof follows from Propositions~\ref{prmz} and~\ref{prmz2}.
\end{proof}


%


\begin{proposition}\label{prmz2}
  Let $L_{\Phi}=L_{\Phi}(\Omega,\mu)$ be an Orlicz space.
  Assume that the points $(x_k)_{k=1}^N$, the sets $(\Omega_k)_{k=1}^N$, and $V_n$ are the same as in Theorem~\ref{thOr}.
%
Then, for all $P_{n}\in \Psi_n$, we have
     \begin{equation}\label{pre2o}
        \bigg\|\sum_{k=1}^N |P_n(x_k)|\chi_{\Omega_k}\bigg\|_{{\Phi}}\le C_3\|P_n\|_{{\Phi}}.
     \end{equation}

\end{proposition}

\begin{proof}
  We have
 \begin{equation}\label{pre1+}
   \begin{split}
       \Phi(P_n(x_k))&=\Phi\(\int_\Omega P_n(t)T^{x_k}V_n(t)d\mu(t)\)\le \Phi\(\int_\Omega |P_n(t)T^{x_k}V_n(t)|d\mu(t)\).
    \end{split}
 \end{equation}
Denote
$$
d\tau(t)=|T^{x_k}V_n(t)|d\mu(t).
$$
Then, by~\eqref{tr} and~\eqref{vp1},
\begin{equation}\label{ZZZ}
  \begin{split}
     1\le \tau(\Omega)=\int_\Omega |T^{x_k}V_n(t)|d\mu(t)\le C\int_\Omega |V_n(t)|d\mu(t)\le cC=C_4.
   \end{split}
\end{equation}
Applying Jensen's inequality together with~\eqref{ZZZ}, we obtain
\begin{equation}\label{pre2+}
  \begin{split}
     \Phi&\(\int_\Omega |P_n(t)T^{x_k}V_n(t)|d\mu(t)\)=\Phi\(\int_\Omega |P_n(t)|d\tau(t)\)\\
     &\le \Phi\(\frac{C_4}{\tau(\Omega)}\int_\Omega |P_n(t)|d\tau(t)\)\le \frac{1}{\tau(\Omega)}\int_\Omega \Phi\(C_4|P_n(t)|\)d\tau(t)\\
     &\le\int_\Omega \Phi\(C_4|P_n(t)|\)|T^{x_k}V_n(t)|d\mu(t).
   \end{split}
\end{equation}
Now, combining~\eqref{pre1+} with~\eqref{pre2+} and using that $T^{x_k}V_n(t)=T^{t}V_n(x_k)$ and~\eqref{VVV}, we obtain
\begin{equation*}
  \begin{split}
     \sum_{k=1}^N \Phi(P_n(x_k))\l_k&\le  \int_\Omega \Phi\(C_4|P_n(t)|\)\sum_{k=1}^N|T^{t}V_n(x_k)|\l_kd\mu(t)\\
     &\le C_2\int_\Omega \Phi\(C_4|P_n(t)|\)d\mu(t),
   \end{split}
\end{equation*}
where $\l_k=\int_{\Omega_k}d\mu(t)$. From the last inequality, we easily obtain~\eqref{pre2o}.
\end{proof}

\begin{remark}
We note that conditions~\eqref{vp1}  and~\eqref{VVV} hold if $T^s$ is admissible, there exists the quadrature formula
\begin{equation*}
        \int_\Omega P(x)d\mu(x)=\sum_{k=1}^N P(x_k)\l_k\quad\text{for all}\quad P\in \Psi_{2n},
\end{equation*}
and the kernel $V_n$ can be represented in the following form
$$
V_n(x)=\sum_{j=1}^K \a_j F_{j,n}(x),
$$
where $F_{j,n}\in \Psi_{2n}$, $\sup_{j,n} \|F_{j,n}\|_{L_{1}(\Omega,\mu)}\le c<\infty$, $F_{j,n}(x)\ge 0$, and $K$ and $\a_j$ do not depend on $n$.

Indeed, property~\eqref{vp1} is clear and~\eqref{VVV} follows from the following relations:
     \begin{equation*}
     \begin{split}
                \bigg\|\sum_{k=1}^N &|T^tV_n(x_k)|\chi_{\Omega_k}\bigg\|_{L_{1}(\Omega,\mu)}=\sum_{k=1}^N |T^tV_n(x_k)|\l_k\\
                &\le \sum_{j=1}^K |\a_j|\sum_{k=1}^N T^{t}F_{j,n}(x_k)\l_k=\sum_{j=1}^K |\a_j|\int_\Omega T^{t}F_{j,n}(u)d\mu(u)\\
                &\le C\sum_{j=1}^K |\a_j| \|F_{j,n}\|_{L_{1}(\Omega,\mu)}\le  cC\sum_{j=1}^K |\a_j|.
     \end{split}
     \end{equation*}
As an example, take the orthonormal trigonometric system $(\psi_k)_{k=0}^\infty$ and consider  the de la Vall\'ee Poussin kernel
$V_n$ given by
$V_n(x)=2K_{2n-1}(x)-K_{n-1}(x)$, where $K_n$ is the Fej\'er kernel. Then $V_n$ satisfies both
\eqref{vp1}  and~\eqref{VVV}.

\end{remark}


\section{MZ inequalities for trigonometric polynomials}

Denote by $\mathcal{T}_n^d$
the set of all trigonometric polynomials  of degree at most $n$ in each variable on $\T^d=[0,2\pi)^d$, 
 i.e.,
$$
\mathcal{T}_n^d:=\bigg\{T\,:\, T(x)=\sum_{k_1=-n}^n\dots \sum_{k_d=-n}^n c_k e^{i(k,x)} \bigg\}.
$$
In the case $d=1$, we denote $\mathcal{T}_n=\mathcal{T}_n^1$ and $\T=\T^1$.

\subsection{Minimal set of nodes}

In the case of minimal set of nodes in the MZ inequality, we obtain the following result, cf.~\eqref{mzclassical}.

\begin{theorem}\label{thF}
Let $X$ be a Banach lattice on $\T$ endowed with the Lebesgue measure, $d\mu(x)=dx$.
If there exists $C>0$ such that
\begin{equation}\label{SnSn}
  \| S_n f\|_{X}\le C\|f\|_{X},\quad f\in X,
\end{equation}
where $S_n$ is the n-th Fourier partial sum operator.
Then, for all $T_{n}\in \mathcal{T}_n$, we have
    \begin{equation}\label{EEE1t}
      \|T_n\|_{X}\asymp\bigg\|\sum_{k=0}^{2n} |T_n(t_k)|\chi_{[t_{k},t_{k+1}]}\bigg\|_{X},\quad t_k=\frac{2\pi k}{2n+1}.
     \end{equation}
\end{theorem}

\begin{proof}
  By Theorem~\ref{thOr2}, it is enough to verify that the following Bernstein inequality holds:
  \begin{equation*}
    \|T_n'\|_X\le 2Cn\|T_n\|_X,\quad T_n\in \mathcal{T}_n.
 \end{equation*}
  It is well known that
  $$
  T_n'(x)=\frac n{\pi}\int_\T T_n(u+x)\sin nu K_{n-1}(u)du,
  $$
  where $K_{n-1}(x)=\sum_{|k|\le n-1} (1-\frac{|k|}{n})e^{ikx}\ge 0$, $x\in \T$, is the classical Fej\'er kernel. Therefore, 
  \begin{equation}\label{KnKn}
    \|T_n'\|_X\le 2n\| |T_n|*K_{n-1}\|_X.
  \end{equation}
  It remains  to note that $f*K_{n-1}=\frac1{n}\sum_{k=0}^{n-1}S_kf$ and to apply the triangle inequality and~\eqref{SnSn}.
\end{proof}

\begin{remark}
    Relation~\eqref{EEE1t} with explicit constants can be written as follows:
    \begin{equation}\label{TTMZ}
              e^{-\frac{4\pi C}{3}}\bigg\|\sum_{k=0}^{2n} |T_n(t_k)|\chi_{[t_{k},t_{k+1}]}\bigg\|_{X} \le \|T_n\|_{X} \le Ce^{\frac{4\pi C}{3}}\bigg\|\sum_{k=0}^{2n} |T_n(t_k)|\chi_{[t_{k},t_{k+1}]}\bigg\|_{X},
    \end{equation}
    where $C$ is given by \eqref{SnSn}.
Additionally, if \eqref{SnSn} as well as the Bernstein inequality
$  \|T_n'\|_X\le B n\|T_n\|_X$ hold, then
    the factor $e^{\frac{4\pi C}{3}}$ can be replaced by $e^{\frac{2\pi B}{3}}$. In particular, if $X$ is translation invariant, then $B=1$.

    Note that $C$ given by \eqref{SnSn} may depend on $n$. For example, take
    $X=L_1(\T)$ and $C\asymp \log n$, cf.~\cite[Ch. X, Th. 7.28]{Z}.


\end{remark}
   \begin{proof}
Indeed, under the conditions of Theorem~\ref{thF},
conditions 1)-4) of Theorem~\ref{thOr2} hold with the following constants:
$c_1=c_2=\b=1$, $C_1=C$, $B_\a=(2C)^\a$, $\a\in \N$, while condition 5) of Theorem~\ref{thOr2} is valid with $A=\frac{2\pi}{3}$ and $\eta=e^{\frac{4\pi C}{3}}-1$.

If $X$ is translation invariant, then $B_\a=1$ and, hence, $\eta=e^{\frac{2\pi}{3}}-1$.
\end{proof}



\begin{example}
    We note that~\eqref{SnSn} holds (that is, the MZ inequality \eqref{EEE1t} is valid) in the following cases:

\begin{itemize}
    \item[$(i)$] $X=L_p(\T)$, $1<p<\infty$, cf.  inequality~\eqref{mzclassical};

    \item[$(ii)$] $X=L_\Phi(\T)$ is the Orlicz space, where $\Phi$ satisfies conditions 1), 2) and 3) and $\Phi, \Phi^*\in\Delta_2$, see Theorem~\ref{th1} below; see also~\cite{PW22} for some MZ inequalities 
     in Orlicz spaces;

\item[$(iii)$]  $X=L_{p,q}(\T)$, $1<p<\infty$, $1\le q\le \infty$ is the Lorentz space;

   \item[$(iv)$]  $X=L_{p(x)}(\T)$ is the variable Lebesgue space  with a periodic function $p(x)$ satisfying the log-condition~\eqref{log} and $1<p_-\le p(x)\le p_+<\infty$, see, e.g.,~\cite[2.8.4]{KMRS16}.
\end{itemize}
At the same time, it is well known that~\eqref{EEE1t} is not valid if $X=L_1$ or $L_\infty$, see~\cite[Ch.~X, \S~7]{Z}.
\end{example}

Let us specify Theorem~\ref{thF} in the case of Orlicz spaces on $\T^d$ endowed with the norm
\begin{equation*}\label{luxnorm2}
  \|f\|_{\Phi}=\inf\left\{\l\,:\, \int_{\T^d}\Phi\(\frac{f(x)}{\l}\)dx\le 1\right\}.
\end{equation*}
Let $X_m=(x_j)_{j\in [1,m]^d}$
be a set of $m^d$ distinct points on $\T^d$. We denote
\begin{equation}\label{n0}
  \|f\|_{\Phi,X_m}
  =\inf \Big\{ \lambda>0 : \(\frac{2\pi}{m}\)^d \sum\limits_{j\in [1,m]^d} \Phi\Big(\frac{f(x_j)}{\lambda}\Big)  \le 1\Big\}.
\end{equation}

In the next theorem, we establish MZ inequalities for Orlicz spaces in their classical form. Special attention is given to the numerical constants. In particular, in Theorem~\ref{th3}, we obtain sharper estimates than those derived from Theorem~\ref{corMZinv}.


\begin{theorem}\label{th1}
Let $\Phi$ be an Orlicz function such that $\Phi\in \Delta_2$ and $\Phi^*\in \Delta_2$. Let $T_n \in \mathcal{T}_n^d$ and $X_{2n+1}=(x_j)_{j\in [0, 2n]^d}$ with $x_j =\frac{2 \pi j }{2n +1}$. Then
$$
\frac1{3^d C_{\Phi,d} } \|T_n\|_{\Phi}\le \|T_n\|_{\Phi, X_{2n+1}}\le 3^d\|T_n\|_{\Phi},
$$
where the constant $C_{\Phi,d}$ is defined in~\eqref{sn} below.
\end{theorem}

\begin{proof}
The estimate from above is a simple corollary of the following inequality
$$
\frac{2\pi}{2n+1}\sum\limits^{2n}_{\nu=0} \Phi\({T_n\(\frac{2 \pi \nu}{2n+1}\)}\)  \le \int_{\T} \Phi\(3 {T_n(x)}\) dx,\quad T_n\in \mathcal{T}_n^1,
$$
where $\Phi$  is a convex non-negative and non-decreasing function on $\mathbb{R_+}$, see~\cite[Ch.~X, \S.~7]{Z}.
The estimate from below follows from Proposition~\ref{prmz}. Let us only show that condition 3) of Proposition~\ref{prmz} holds under the conditions of Theorem~\ref{th1}.
Indeed, using the fact that  $\|S_ng\|_{L_p(\T^d)} \le C_{p,d} \|g\|_{L_p(\T^d)}$, $1<p<\infty$, where
$$
S_ng(x)=\sum_{k\in [-n,n]^d} \widehat{g}(k)e^{i(k,x)},
$$
see, e.g.,~\cite[Theorem~4.2.5]{W12} and~\cite[Theorem 3]{R63}, we get that
\begin{equation}\label{sn}
  \|S_ng\|_{\Phi} \le C_{\Phi,d} \|g\|_{\Phi}
\end{equation}
holds provided $L_{\Phi^*}$ is reflex. At the same time, by~\cite[p.~226]{KR61} the space $L_{\Phi^*}$ is reflex if and only if  $\Phi^*$ and $\Phi$ satisfy $\Delta_2$-condition.
\end{proof}

\subsection{MZ type inequalities in Banach lattices}


Here we start with the estimates for maximal and minimal functions.


\begin{proposition}\label{corn1}

Let $X=X(\T^d,\mu)$ be a Banach function lattice such that
\begin{equation}\label{BerBL}
  \left\|\frac{\partial}{\partial x_i}T_n\right\|_X\le Bn\|T_n\|_X,\quad i=1,\dots,d,\quad T_n\in \mathcal{T}_n^d,
\end{equation}
for some constant $B>0$. Then, for any $0<A<\frac{\log 2}{dB}$ and any $T_n\in \mathcal{T}_n^d$, we have
  \begin{equation*}
    \frac{1}{e^{dAB}}\Big\|\max_{t\in Q(x,A/n)}|T_n(t)|\Big\|_X\le \|T_n\|_X\le \frac{1}{2-e^{dAB}}\Big\|\min_{t\in Q(x,A/n)}|T_n(t)|\Big\|_X.
  \end{equation*}

\end{proposition}

\begin{proof}

This is a direct corollary from Theorem~\ref{thmmwe}.
\end{proof}

\begin{remark}\label{remF}
Instead of assuming the Bernstein inequality~\eqref{BerBL}, we can suppose that there exists a constant $c>0$ such that
\begin{equation}\label{Fejer}
    \|\s_{n,i}f\|_X\le c\|f\|_X,\quad f\in X,
\end{equation}
where $\s_{n,i}f$ is the classical F\'ejer means of $f$ with respect to the variable $x_i$. This can be shown by repeating the arguments of the proof of Theorem~\ref{thF}.
\end{remark}


Proposition~\ref{corn1}  implies the following result.

\begin{theorem}\label{cheT}
  Let $X=X(\T^d,\mu)$ be a Banach lattice and let $\e>0$. Assume that~\eqref{BerBL} holds with the constant $B>0$.
  Then, for all $T_n\in \mathcal{T}_n^d$ and $N\ge (\frac{2\pi dB}{\log 2}+\e)n$,
\begin{equation}\label{TrMZ}
  \|T_n\|_X\asymp \bigg\|\sum_{k_1=1}^{N}\dots\sum_{k_d=1}^{N} |T_n(x_k)|\chi_{\Omega_k}\bigg\|_{X},
\end{equation}
  where $x_k$ is an arbitrary point in $\Omega_k=[\frac{2\pi(k_1-1)}{N},\frac{2\pi k_1}{N}]\times\dots\times [\frac{2\pi(k_d-1)}{N},\frac{2\pi k_d}{N}]$, $k\in [1,N]^d$.
\end{theorem}

%
%

\begin{example}\label{exbb}
The MZ inequality \eqref{TrMZ}
holds in the following cases:
\begin{itemize}
    \item[$(i)$] $X$ is any translation invariant Banach lattice on $\T^d$. For such spaces $X$, equivalence~\eqref{TrMZ} holds with $N\ge (\frac{2\pi d}{\log 2}+\e)n$. In particular, as a space $X$, we can consider the mixed Lebesgue spaces
$X=L_{\overline{p}}(\T^d)$, $\overline{p}=(p_1,\dots,p_d)$, $1\le p_i\le \infty$, $i=1,\dots,d$, see~\cite[p.~3.3.4]{Tem};   the Lorentz spaces $L_{p,q}(\T^d)$ with $1\le p\le \infty$; the  Orlicz spaces $X=L_{\Phi}(\T^d)$, where the function $\Phi$ satisfies only conditions 1) and 2) of Definition~\ref{defO}; the Morrey spaces $M_{p,q}^\l(\T^d)$ with $1\le p<\infty$, $1\le q\le \infty$, and $0\le \l\le \frac dp$.

\item[$(ii)$]  $X=L_{p,w}(\T)$, $1\le p\le \infty$, where $w$ is a doubling weight (see~\cite{MT00}) or a non-doubling weight satisfying certain  conditions (see \cite{BT17} for a needed Bernstein inequality~\eqref{BerBL});

\item[$(iii)$]  $X=L_{p(x)}$ is a variable Lebesgue space,  where $p(x)$ is such that $1\le p_-\le p(x)\le p_+<\infty$ and the log-condition~\eqref{log} holds, see, e.g.,~\cite{W23} for boundedness of the Fej\'er means~\eqref{Fejer} in this case.
\end{itemize}
\end{example}


\subsection{MZ type inequalities in quasi-Banach lattices}

The next theorem follows directly from Theorem~\ref{corthmmwe}.

\begin{theorem}\label{qB}
  Let $X=X(\T^d,\mu;q)$, $0<q\le 1$, be a quasi-Banach lattice.   Assume that, for all $\a\in\Z_+^d$, there exist $B_\a>0$ and $A>0$ such that
  \begin{equation}\label{berweQu}
    \|D^\a T_n\|_X\le B_\a n^{|\a|_1}\| T_n\|_X,\quad T_n\in \mathcal{T}_n,
  \end{equation}
and $\sum_{\a\in\Z_+^d\setminus  \{\bar{0}\}}\(\frac{A^{|\a|_1}B_\a}{\a!}\)^{q}<1$. Then, for all $T_n\in \mathcal{T}_n^d$ and  $N\ge \frac{2\pi}{A}n$,
\begin{equation}\label{TrMZ++++}
  \|T_n\|_X\asymp \bigg\|\sum_{k_1=1}^{N}\dots\sum_{k_d=1}^{N} |T_n(x_k)|\chi_{\Omega_k}\bigg\|_{X},
\end{equation}
  where $x_k$ is an arbitrary point in $\Omega_k=[\frac{2\pi(k_1-1)}{N},\frac{2\pi k_1}{N}]\times\dots\times [\frac{2\pi(k_d-1)}{N},\frac{2\pi k_d}{N}]$, $k\in [1,N]^d$.
\end{theorem}


\begin{example}
The MZ inequality \eqref{TrMZ++++}
holds in the following cases:
    \begin{itemize}
        \item[$(i)$] $X=L_{\overline{p}}(\T^d)$, $\overline{p}=(p_1,\dots,p_d)$, $0<p_i\le \infty$, $i=1,\dots,d$. Note that if $p_1=\dots=p_d=p$, we obtain the MZ inequalities in the classical case of $L_p(\T^d)$, $0<p\le\infty$, cf.~\eqref{mzclassical};

    \item[$(ii)$] $X=L_{p,w}(\T)$, $0<p\le \infty$, where $w$ is a doubling or non-doubling weight satisfying certain natural conditions, see Example~\ref{exbb}(ii) for the case $p\ge 1$ and~\cite{E99} and~\cite{BT17} for
    the case $0<p<1$.
    \end{itemize}
\end{example}


\subsection{More Marcinkiewicz--Zygmund inequalities in Orlicz spaces}
In this section, we specify some of the above MZ inequalities to the case of the Orlicz spaces $L_\Phi=L_\Phi(\T)$ endowed with the Lebesgue measure, by tracking the corresponding constants and  considering non-uniform nodes.

We  suppose that the function $\Phi$ satisfies only conditions 1) and 2) of Definition~\ref{defO}.
Note that under these conditions $\Phi$ is not necessarily an Orlicz function (take $\Phi(t)=t$). In what follows, together with the Luxemburg norm $\|f\|_{\Phi}$, it will be convenient to use the following Orlicz norm, see~\cite[p.~222]{KR61}:
\begin{equation}\label{n1}
\|f\|^\sharp_{\Phi,d\mu}=\inf_{\kappa>0}\left\{\frac{1}{\kappa}\(1+\int_{\T}\Phi(\kappa f(t))d\mu(t)\) \right\}.
\end{equation}
If $d\mu(t)=dt$, then we will simple write $\|f\|^\sharp_{\Phi,d\mu}=\|f\|^\sharp_{\Phi}$.
It is well known (see, e.g., ~\cite[\S~9, (9.24)]{KR61}) that,
 for any $f\in L_{\Phi}(\T)$,
 \begin{equation}\label{4}
     \|f\|_{\Phi}\le \|f\|^\sharp_{\Phi}\le 2\|f\|_{\Phi}.
 \end{equation}

We will also need the following simple fact, which directly follows from~\eqref{n1}.
\begin{lemma}\label{le1}
 Let functions $f, g$ and measures $\mu_1$ and $\mu_2$ be such that, for any $\kappa>0$,
 \begin{equation*}
   \int_{\T} \Phi(\kappa f(t))d\mu_1 \le C \int_{\T} \Phi(\kappa g(t))d\mu_2,
 \end{equation*}
 where $C$ is some positive constant. Then
 \begin{equation*}
   \|f\|^\sharp_{\Phi,d\mu_1}\le\max\{C,1\}\|f\|^\sharp_{\Phi,d\mu_2}.
 \end{equation*}
 \end{lemma}


Similarly as in~\eqref{n0}, for $X_m=(x_j)_{j=1}^m\subset \T$ and $(a_j)_{j=1}^m\subset \C$, we denote
\begin{equation*}\label{n2}
  \|f\|^\sharp_{\Phi,X_m}
  =\inf_{\kappa>0} \left\{\frac1{\kappa}\(1+\frac{2\pi}{m} \sum_{j=1}^m \Phi \big(\kappa f(x_j)\big)\)\right\}
\end{equation*}
and
\begin{equation*}
  \|(a_j)\|^\sharp_{\Phi,m}=\inf_{\kappa>0} \left\{\frac1{\kappa}\(1+\frac{2\pi}{m} \sum_{j=1}^m  \Phi(\kappa a_j)\)\right\}.
\end{equation*}

We will make use of the following result, which improves Theorem~2 in~\cite{LMN87}.

\begin{lemma}\label{le2} {\sc (See~\cite[Theorem 2.3]{Jo99}.)}
    Let $\Phi$ be an even, convex, non-negative, and non-decreasing function on $[0,\infty)$. Let $(\t_k)_{k=1}^m$, $m\in \N$, be such that
  $0\le \t_1<\t_2<\dots<\t_m<2\pi$ and $\d=\min_k(\t_{k+1}-\t_k)$, where $\t_{m+1}=2\pi+\t_1$. Then
\begin{equation}\label{8}
 \sum_{k=1}^m \Phi(T_n(\t_k))\le \(\frac{n+1}{2\pi}+\frac1\d\)\int_\T \Phi(eT_n(t)) dt,\quad T_n\in \mathcal{T}_n.
  \end{equation}
\end{lemma}


We will also need the  Bernstein inequality
\begin{equation}\label{berOLD}
  \|T_n'\|_{\Phi} \le n\|T_n\|_{\Phi}, \quad T_n\in \mathcal{T}_n,
\end{equation}
which immediately follows from the classical Bernstein inequality $\|T_n'\|_{\infty} \le n\|T_n\|_{\infty}$ and Remark~\ref{remm2.5}.




\begin{theorem}\label{th3}
Let $\Phi$ be an even, convex, non-negative, and non-decreasing function on $[0,\infty)$. Suppose that $X_N=(\t_k)_{k=1}^N$, $N\in \N$, be such that
  $0\le \t_1<\t_2<\dots<\t_N<2\pi$
  and
  $$
  \d=\min_k(\t_{k+1}-\t_k),
  $$
  $$
  \l=\max_k(\t_{k+1}-\t_k),
  $$
  where $\t_{N+1}=2\pi+\t_1$. Then
    \begin{equation}\label{9}
      \|T_n\|^\sharp_{\Phi,X_N}\le \frac{e(n+1+\frac{2\pi}\d)}N\|T_n\|^\sharp_{\Phi},\quad T_n\in \mathcal{T}_n.
  \end{equation}
If, additionally,   $N \min\left\{1,\frac{4\pi}{\l N}\right\} >4e\l n(n+1+\frac{2\pi}{\d})$, then
    \begin{equation}\label{10}
      \|T_n\|^\sharp_{\Phi}\le \frac{2}{\min\left\{1,\frac{4\pi}{\l N}\right\}-{4e\l n(n+1+\frac{2\pi}{\d})}N^{-1}}\|T_n\|^\sharp_{\Phi,X_N},\quad T_n\in \mathcal{T}_n.
  \end{equation}
\end{theorem}

\begin{proof}
  First, we prove inequality~\eqref{9}. Applying~\eqref{8} and Lemma~\ref{le1} for $f=T_n$, $g=eT_n$, and
  $$
  \mu_1(x)=\frac{2\pi}{N}\sum_{j=1}^N\delta(x-\t_j),\quad \mu_2(x)=x,
  $$
  we have
  \begin{equation*}
    \|T_n\|^\sharp_{\Phi,X_N}=\|T_n\|^\sharp_{\Phi,d\mu_1}\le \max\left\{\frac{n+1+\frac{2\pi}\d}N,1\right\}\|e T_n\|^\sharp_{\Phi},
  \end{equation*}
  which, in view of the norm homogeneity and the inequality $\d\le\frac{2\pi}{N}$, implies~\eqref{9}.

  Second, we show~\eqref{10}. Without loss of generality, we may assume that $N$ is even. Denote
  $$
  M'=\bigcup_{j=1}^{N/2} [\t_{2j-1},\t_{2j}], \quad
  M''=\bigcup_{j=1}^{N/2} [\t_{2j},\t_{2j+1}].
  $$
  We have
  $$
  T_n(x)=T_n(x)\chi_{M'}(x)+T_n(x)\chi_{M''}(x),
  $$
  which gives
  \begin{equation}\label{11}
    \|T_n\|^\sharp_{\Phi}\le \|T_n\chi_{M'}\|^\sharp_{\Phi}+\|T_n\chi_{M''}\|^\sharp_{\Phi}.
  \end{equation}
  Let $\eta_{2j-1}\in [\t_{2j-1},\t_{2j}]$ be such that
  $$
  |T_n(\eta_{2j-1})|=\max_{t\in [\t_{2j-1},\t_{2j}]}|T_n(t)|,\quad j=1,\dots,\frac N2.
  $$
  Then, taking into account
  Lemma~\ref{le2}, we have
  \begin{equation}\label{12}
    \begin{split}
      \|T_n\chi_{M'}\|^\sharp_{\Phi}&=
      \inf_{\kappa>0}\left\{\frac{1}{\kappa}\(1+\sum_{j=1}^{N/2}
      \int_{\t_{2j-1}}^{\t_{2j}}\Phi(\kappa T_n(t))dt\) \right\}\\
      &\le \inf_{\kappa>0}\left\{\frac{1}{\kappa}\(1+\sum_{j=1}^{N/2}
      \Phi(\kappa T_n(\eta_{2j-1})) (\t_{2j}-\t_{2j-1})\) \right\}\\
      &\le \inf_{\kappa>0}\left\{\frac{1}{\kappa}\(1+\l\sum_{j=1}^{N/2}
      \Phi(\kappa T_n(\eta_{2j-1}))\) \right\}\le \max\{1,\tfrac{\l N}{4\pi}\}\|T_n\|^\sharp_{\Phi,Y'_{N/2}},
    \end{split}
  \end{equation}
  where $Y'_{N/2}=(\eta_{2j-1})_{j=1}^{N/2}$. Similarly, we obtain that
  \begin{equation}\label{13}
    \|T_n\chi_{M''}\|^\sharp_{\Phi}\le \|T_n\|^\sharp_{\Phi,Y''_{N/2}},
  \end{equation}
  where $Y''_{N/2}=(\eta_{2j})_{j=1}^{N/2}$ with $\eta_{2j}\in [\t_{2j},\t_{2j+1}]$  such that
  $|T_n(\eta_{2j})|=\max_{t\in [\t_{2j},\t_{2j+1}]}|T_n(t)|,\, j=1,\dots,N/2$.
  Denoting $X'_{N/2}=(\t_{2j-1})_{j=1}^{N/2}$, $X''_{N/2}=(\t_{2j})_{j=1}^{N/2}$ and   combining~\eqref{11}, \eqref{12}, and \eqref{13}, we obtain
  \begin{equation}\label{14}
    \begin{split}
       \tfrac{1}{\max\left\{1,\tfrac{\l N}{4\pi}\right\}}\|T_n\|^\sharp_{\Phi}-&\(\|T_n\|^\sharp_{\Phi,X'_{N/2}}+\|T_n\|^\sharp_{\Phi,X''_{N/2}}\)\\
       &\le
       \(\|T_n\|^\sharp_{\Phi,Y'_{N/2}}-\|T_n\|^\sharp_{\Phi,X'_{N/2}}\)+
       \(\|T_n\|^\sharp_{\Phi,Y''_{N/2}}-\|T_n\|^\sharp_{\Phi,X''_{N/2}}\):=I_1+I_2.
    \end{split}
  \end{equation}
  Next, by the mean value theorem there exists a set of points $Z'_{N/2}=(\xi_{2j-1})_{j=1}^{N/2}$, $\xi_{2j-1}\in [\tau_{2j-1},\eta_{2j-1}]$, such that
  \begin{equation}\label{15}
    \begin{split}
      I_1&\le \|(T_n(\eta_{2j-1})-T_n(\t_{2j-1}))\|^\sharp_{\Phi,N/2}
        =\|T_n'(\xi_{2j-1})(\eta_{2j-1}-\t_{2j-1})\|^\sharp_{\Phi,N/2}\\
        &\le\l\|T_n'\|^\sharp_{\Phi,Z'_{N/2}}
        \le \frac{2e\l (n+1+\frac{2\pi}{\d'})}N
        \|T_n'\|^\sharp_{\Phi},
    \end{split}
  \end{equation}
  where the last inequality follows from~\eqref{9} with
  $\d'=\min_{j=1,\dots,N/2}\{\xi_{2j+1}-\xi_{2j-1}\}$. Then, applying the Bernstein inequality~\eqref{berOLD} to the right-hand side of~\eqref{15}, we obtain
  \begin{equation}\label{16}
    \begin{split}
      I_1
        \le \frac{2e\l n(n+1+\frac{2\pi}{\d})}N
        \|T_n\|^\sharp_{\Phi}.
    \end{split}
  \end{equation}
  Similarly, we get
    \begin{equation}\label{17}
    \begin{split}
      I_2
        \le \frac{2e\l n(n+1+\frac{2\pi}{\d})}N
        \|T_n\|^\sharp_{\Phi}.
    \end{split}
  \end{equation}
Now, combining~\eqref{14}, \eqref{16}, and \eqref{17}, we derive
  \begin{equation*}
    \begin{split}
      \({\min\left\{1,\frac{4\pi}{\l N}\right\}}-\frac{4e\l
 n(n+1+\frac{2\pi}{\d})}N\)\|T_n\|^\sharp_{\Phi}
      \le \|T_n\|^\sharp_{\Phi,X'_{N/2}}+\|T_n\|^\sharp_{\Phi,X''_{N/2}}\le 2\|T_n\|_{\Phi,X_N}^\sharp,
    \end{split}
  \end{equation*}
  which proves~\eqref{10}.
\end{proof}

We present one more MZ inequalities for $2n+1$ non-equidistant nodes.
%
\begin{theorem}\label{th4}
 Let $\Phi$ be an even, convex, non-negative, and non-decreasing function on $[0,\infty)$.  Let  $Y_{2n+1}=(y_j)_{j=0}^{2n}$ be such that
  $$
  \frac{2\pi(j-\s)}{2n+1}\le y_j <\frac{2\pi(j+\s)}{2n+1},\quad j=0,\dots,2n,
  $$
for some $\s\in (0,1/4)$. Then
  \begin{equation}\label{18}
    \|T_n\|^\sharp_{\Phi,Y_{2n+1}}\le \frac{8e}{3}\|T_n\|^\sharp_{\Phi},\quad T_n\in \mathcal{T}_n.
  \end{equation}
  If, additionally, $\Phi$ is an Orlicz function satisfying $\Phi\in \Delta_2$ and $\Phi^*\in \Delta_2$, and $\s>0$ is such that  $\s<\frac1{16 e\pi C_\Phi}$, where $C_\Phi$ is defined in~\eqref{sn}, then
  \begin{equation}\label{19}
   \frac13\(\frac1{2C_\Phi}-8e\pi\s\) \|T_n\|^\sharp_{\Phi}\le \|T_n\|^\sharp_{\Phi,Y_{2n+1}},\quad T_n\in \mathcal{T}_n.
  \end{equation}
\end{theorem}

\begin{proof}
  The proof of inequality~\eqref{18} follows directly from~\eqref{9} with $N=2n+1$.  Indeed, for $\s\le 1/4$, we have $\d\ge \frac{2\pi}{2n+1}(1-2\s)\ge \frac{\pi}{2n+1}$. Hence,~\eqref{9} gives
  \begin{equation*}
    \begin{split}
         \|T_n\|^\sharp_{\Phi,2n+1}\le \frac{e(n+1+2(2n+1))}{2n+1}\|T_n\|^\sharp_{\Phi}\le \frac{8e}{3}\|T_n\|^\sharp_{\Phi}.
     \end{split}
  \end{equation*}

  Now we show~\eqref{19}. Let $X_{2n+1}=(x_j)_{j=0}^{2n}$, where $x_j =\frac{2 \pi j }{2n +1}$. By Theorem~\ref{th1} and~\eqref{4}, we have
  \begin{equation}\label{20}
    \begin{split}
      \frac1{6C_\Phi}\|T_n\|_{\Phi}^\sharp &\le  \frac1{3C_\Phi}\|T_n\|_{\Phi }\le \|T_n\|_{\Phi,X_{2n+1}}\\
      &\le \|(T_n(x_j)-T_n(y_j))\|_{\Phi,{2n+1}}+\|T_n\|_{\Phi,Y_{2n+1}}.
    \end{split}
  \end{equation}
  By the mean value theorem, there exists a set of points $Z_{2n+1}:=(\xi_{j})_{j=0}^{2n}$, $\xi_{j}\in [\min(x_j,y_j), \max(x_j,y_j)]$, such that
  \begin{equation}\label{21}
    \begin{split}
      \|(T_n(x_j)-T_n(y_j))\|_{\Phi,{2n+1}}=\|(T_n'(\xi_j)|x_j-y_j|)\|_{\Phi,{2n+1}}\le \frac{2\pi\s}{2n+1}\|T_n'\|_{\Phi,Z_{2n+1}}.
    \end{split}
  \end{equation}
 Next, by~\eqref{9} with $N=2n+1$  and
 $$
 \d=\min\{\xi_{j+1}-\xi_j\}\ge\frac{2\pi(1-2\s)}{2n+1},
 $$
 we get
  \begin{equation*}
    \begin{split}
      \|T_n'\|_{\Phi,Z_{2n+1}}&\le
      \frac{e(n+1+\frac{2n+1}{1-2\s})}{2n+1} \|T_n'\|_{\Phi}\le \frac{8e}{3}\|T_n'\|_{\Phi}.
    \end{split}
  \end{equation*}
  This, together with~\eqref{21} and the Bernstein inequality~\eqref{berOLD}, gives
  \begin{equation}\label{22}
    \|(T_n(x_j)-T_n(y_j))\|_{\Phi,{{2n+1}}}\le \frac{2\pi\s n}{2n+1}\cdot\frac{8e}{3}\|T_n\|_{\Phi}
    \le \frac{8 e\pi \s}{3}\|T_n\|_{\Phi}.
  \end{equation}
Finally, combining~\eqref{20} and~\eqref{22}  and applying~\eqref{4}, we obtain~\eqref{19}.
\end{proof}


\subsection{Historical notes}
In the classical case $X=L_p(\T^d)$ as well as for the mixed Lebesgue spaces $L_{\overline{p}}(\T^d)$, MZ inequality~\eqref{TrMZ} with $1<p<\infty$ is well known,  see, e.g.,~\cite[Ch.~X, \S~7]{Z} and \cite[3.3.18]{Tem}. For the case $0<p<1$, see \cite{Pe83} and~\cite{RS98}.

When $X=L_\Phi(\T)$ is the Orlicz space, relation~\eqref{EEE1t} was proved in~\cite{PW22} under the condition that the Hilbert transform is bounded in $X$. One sided MZ inequalities in the Orlicz spaces follow also directly from the results of the work~\cite{LMN87}.
See also \cite{kosov} for the recent sampling discretization results in Orlicz norms on finite dimensional spaces.

In the case of the weighted spaces $X=L_{p,w}(\T)$, $0<p<\infty$, Mastroianni and Totik~\cite{MT00} and Erd\'elyi~\cite{E99} proved~\eqref{TrMZ} for doubling weights $w$.
Necessary and sufficient conditions for MZ inequalities in $L_p(\T)$ to hold were obtained by Chui et al. in~\cite{CZ99}, \cite{CZ93} and by Ortega-Cerd\'a and Saludes in~\cite{OP07}. Some versions of MZ inequalities for trigonometric polynomials in $L_p(\T^d)$ spaces were also obtained in~\cite{Jo99}, \cite{KKLT22}, \cite{O86}, \cite{Xu91}, \cite{RS97}, \cite{Kr20}.

\section{MZ inequalities for algebraic polynomials}\label{seca}

By $\mathcal{P}_n$ we denote the set of all algebraic polynomials of degree at most $n$,
$$
\mathcal{P}_n=\bigg\{P\,:\, P(x)=\sum_{k=0}^n a_k x^k,\,\, a_k\in \C\bigg\}.
$$

\subsection{MZ inequalities via Markov's inequality}
The well-known Markov inequality states that for each algebraic polynomial $P_n\in \mathcal{P}_n$ and any interval $(a,b)\subset \R$ there holds
\begin{equation}\label{Mark}
  \|P_n'\|_{L_\infty[a,b]}\le \frac{2n^2}{b-a}\|P_n\|_{L_\infty[a,b]}.
\end{equation}
First, we derive a generalization of this inequality to the case of an arbitrary rearrangement invariant Banach lattice.

\begin{proposition}\label{pMarkX}
  Let $X$ be a rearrangement invariant Banach lattice on $[a,b]\subset \R$ and let $r\in \N$. Then, for each $P_n\in \mathcal{P}_n$,  there holds
\begin{equation*}
  \|P_n^{(r)}\|_{X}\le \(\frac{8n^2}{b-a}\)^r\|P_n\|_{X}.
\end{equation*}
\end{proposition}

\begin{proof}
We follow the ideas from~\cite{E00}. It is enough to prove the proposition for $r=1$. Let $c=\frac{a+b}{2}$. Applying Lemma~\ref{ilf} for $\Phi_n=\mathcal{P}_n$, $Q=[a,c]$, and $L(s)=s'(c)$, $s\in \mathcal{P}_n$, we can find the numbers $(\l_i)_{i=1}^{2n-1}$ and the points $(x_i)_{i=1}^{2n-1}\subset [a,c]$ such that
\begin{equation}\label{intP}
  L(s)=s'(c)=\sum_{i=1}^{2n-1}\l_i s(x_i),\quad s\in \mathcal{P}_n.
\end{equation}
Moreover, we deduce from~\eqref{Mark} that
\begin{equation}\label{mark0}
\|L\|=\sum_{i=1}^{2n-1}|\l_i|\le \frac{2n^2}{c-a}=\frac{4n^2}{b-a}.
\end{equation}
We set $s(t)=P_n(t-c+x)\in \mathcal{P}_n$ for any $x\in [a,b]$. Using~\eqref{intP} and taking into account that $s'(c)=P_n'(x)$, we obtain
\begin{equation}\label{mark1}
  \begin{split}
     \bigg\|\frac{b-a}{4n^2}P_n'\chi_{[c,b]}\bigg\|_X\le \frac{b-a}{4n^2} \sum_{i=1}^{2n-1}|\l_i \|P_n(x_i-c+\cdot)\chi_{[c,b]}\|_X.
  \end{split}
\end{equation}
It is not difficult to see that the functions $f(x)=P_n(x_i-c+x)\chi_{[c,b]}(x)$ and $g(x)=P_n(x)\chi_{[x_i, x_i-c+b]}(x)$
have equal distribution functions, i.e., $\mu_f=\mu_g$. Therefore,
\begin{equation}\label{mark2}
  \|P_n(x_i-c+\cdot)\chi_{[c,b]}\|_X=\|P_n\chi_{[x_i, x_i-c+b]}\|_X\le \|P_n\|_X.
\end{equation}
Thus, combining~\eqref{mark0}, \eqref{mark1} and~\eqref{mark2}, we obtain
\begin{equation}\label{mark3}
  \|P_n'\chi_{[c,b]}\|_X\le \frac{4n^2}{b-a}\|P_n\|_X.
\end{equation}
Similarly, applying Lemma~\ref{ilf} with $Q=[c,b]$, we have
\begin{equation}\label{mark4}
  \|P_n'\chi_{[a,c]}\|_X\le \frac{4n^2}{b-a}\|P_n\|_X.
\end{equation}
Finally, inequalities~\eqref{mark3} and~\eqref{mark4} yield
\begin{equation*}
  \|P_n'\|_X \le \|P_n'\chi_{[a,c]}\|_X+\|P_n'\chi_{[c,b]}\|_X\le \frac{8n^2}{b-a}\|P_n\|_X.
\end{equation*}
\end{proof}

Now with the help of Proposition~\ref{pMarkX}, we obtain the following result.
\begin{theorem}\label{mzMark}
  Let $X$ be a rearrangement invariant Banach lattice on $[a,b]$ and let $\e>0$. Then, for each $P_n\in \mathcal{P}_n$ and $N\ge (\frac8{\log 2}+\e)n^2$,  we have
\begin{equation*}\label{E1}
  \|P_n\|_X\asymp \bigg\|\sum_{k=1}^{N}|P_n(x_k)|\chi_{\Omega_k}\bigg\|_{X},
\end{equation*}
where $x_k$ is an arbitrary point in $\Omega_k=[a+\frac{(b-a)(k-1)}{N}, a+\frac{(b-a)k}{N}]$, $k=1,\dots,N$.
\end{theorem}

\begin{proof}
  We apply Theorem~\ref{corthmmwe} with $\mathcal{F}_n=\mathcal{P}_n$, $\vp_n\equiv 1$, and $\Omega_k^*=\Omega_k$, $k=1,\dots,N$. Note that by Proposition~\ref{pMarkX}, condition~\eqref{asumpeta} holds if $0<A<\frac{b-a}8\log 2$ and~\eqref{zvv} is fulfilled if $N\ge\frac{n^2}{A}$.
\end{proof}

This theorem extends the main results in~\cite{Sh97} to the case of an arbitrary rearrangement invariant Banach lattice $X$ and also improves the corresponding results to the case of non-regular grids.

\subsection{MZ inequalities via Bernstein's inequality}

\begin{proposition}\label{E0}
Let $X=X([-1,1],\mu;q)$, $0<q\le 1$, be a quasi-Banach lattice, $\vp_n(x)=\sqrt{1-x^2}+\frac1n$, $A>0$, $\d=\tfrac An$ and the points
$$
-1=x_0<x_1<\dots<x_N=1
$$
be such that
\begin{equation}\label{zv}
x_{k+1}-x_k\le \d(\vp_n(x_k)+\vp_n(x_{k+1})),\quad k=0,\dots,N-1,
\end{equation}
and
$$
-\frac{1-\d^2}{{1+\d^2}}\le x_1,\quad x_{N-1}\le \frac{1-\d^2}{{1+\d^2}}.
$$
Assume that for each $r\in \N$ there exists $B_r>0$  such that
$$
\|\vp_n^r P_n^{(r)}\|_X\le B_rn^r \|P_n\|_X,\quad P_n\in \mathcal{P}_n,
$$
and
$$
\sum_{r=1}^\infty\(\frac{A^{r}B_r}{r!}\)^{q}<1.
$$
Then, for any  $P_n\in \mathcal{P}_n$,
  \begin{equation*}
  \begin{split}
     \|P_n\|_X\asymp \bigg\|\sum_{k=0}^{N-1} |P_n(\xi_k)|\chi_{[x_k,x_{k+1}]}\bigg\|_X,
  \end{split}
  \end{equation*}
  where $\xi_k\in [x_{k+1}-\d\vp_n(x_{k+1}),x_k+\d\vp_n(x_k)]$ for $k=0,\dots,N-1$.
\end{proposition}

\begin{proof}
We apply Theorem~\ref{corthmmwe} with
$$\Omega_k=[x_{k},x_{k+1}]\quad\text{and}\quad\Omega_k^*=[x_{k+1}-\d\vp_n(x_{k+1}),x_k+\d\vp_n(x_k)].$$
It is enough to verify~\eqref{zvv}, that is, 
\begin{equation}\label{ineqineq}
  x-\d\vp_n(x)\le x_{k+1}-\d\vp_n(x_{k+1})\le x_k+\d\vp_n(x_k)\le x+\d\vp_n(x)
\end{equation}
for any $x\in (x_k,x_{k+1})$, $k=0,\dots,N-1$. First, we note that the function $\xi(x)=x-\d\vp_n(x)$ is increasing on $(-\frac{1}{\sqrt{1+\d^2}},1)$  and decreasing on $(-1,-\frac{1}{\sqrt{1+\d^2}})$, while the function $\eta(x)=x+\d\vp_n(x)$ is increasing on $(-1,\frac{1}{\sqrt{1+\d^2}})$ and decreasing on $(\frac{1}{\sqrt{1+\d^2}},1)$. Using these properties and~\eqref{zv}, it is not difficult to see that~\eqref{ineqineq} indeed holds for each $k=1,\dots,N-2$. Now we verify~\eqref{ineqineq} for $k=0$ and $k=N-1$. Consider the case $k=N-1$. By the monotonicity of $\xi$ on $(x_{N-1},1)$, we have that $x-\d\vp_n(x)\le x_N-\d\vp_n(x_N)$. Now we claim that
\begin{equation}\label{xxxyyy}
  x_{N-1}+\d\vp_n(x_{N-1})\le x+\d\vp_n(x)
\end{equation}
for $x\in (x_{N-1},1)$. Indeed, if $x\in (x_{N-1},\tfrac1{\sqrt{1+\d^2}})$, then  $\eta$ is increasing and~\eqref{xxxyyy} is valid. If $x\in (\tfrac1{\sqrt{1+\d^2}},1)$, then~\eqref{xxxyyy} holds if and only if $x_{N-1}\le \frac{1-\d^2}{1+\d^2}\le \frac1{\sqrt{1+\d^2}}$. Thus,~\eqref{ineqineq} holds for all $x\in \Omega_{N-1}$. Similarly, one can establish~\eqref{ineqineq} for $x\in \Omega_0$.  Summarizing, we have
$$
\Omega_k^*\subset (x-\d\vp_n(x),x+\d\vp_n(x))\quad\text{for any}\quad x\in \Omega_k,\quad k=0,\dots,N-1,
$$
completing the proof.
\end{proof}

\begin{example}
    Consider  $X=L_{p,w}[-1,1]$, $0<p<\infty$, with a positive weight $w$
such that there holds
\begin{equation*}
  \|\(\sqrt{1-x^2}+\tfrac1n\)^r P_n^{(r)}\|_{L_{p,w}[-1,1]}\le C n^r\|P_n\|_{L_{p,w}[-1,1]},\quad P_n\in \mathcal{P}_n,\quad r\in\N.
\end{equation*}
If $x_j=\cos t_j$, $t_j=\frac{(N-j)\pi}{N}$, $j=0,\dots,N$, then according to Proposition~\ref{E0}, we have
\begin{equation}\label{E2}
  \int_{-1}^1|P_n(x)|^pw(x)dx\asymp\sum_{j=0}^N|P_n(x_j)|^p a_j,
\end{equation}
  where  $N\ge\l n$ for some $\l>1$ and
  $$
  a_j=\int_{t_j-\frac1N}^{t_j+\frac1N}w(\cos t)|\sin t|dt,\quad j=0,\dots,N.
  $$
In the unweighted case, that is, $w\equiv 1$, we have that
$$
a_j=\int_{t_j-\frac1N}^{t_j+\frac1N}|\sin t|dt\sim\frac1{N^2}+\frac{\sin\frac{j\pi}{N}}{N},\quad j=0,\dots,N,
$$
see, e.g.,~\cite{MT00} and~\cite{DMK18}.

The result in~\eqref{E2} was known~$($\cite{MT00} and~\cite{E99}$)$ for doubling weights $w$. Using~\cite{BT17}, we have the same results for some non-doubling weights.
\end{example}

\subsection{Minimal set of nodes}

In the next result,
$\mu$ is defined to be a non-decreasing bounded function on $[a,b]$, $-\infty\le a<b\le+\infty$, which takes infinitely many distinct values and
satisfies $\int_a^b x^k d\mu(x)<\infty$ for each $k\in \Z_+$. Let also $(\psi_n)_{n=0}^\infty$ be a set of orthonormal polynomials with respect to $\mu$ and let $S_n$ be the $n$-th Fourier partial  sum operator with respect to this system.

\begin{theorem}\label{algw}
Let $X=X([a,b],\mu)$ be a Banach function lattice and let $(x_k)_{k=1}^n$ be zeros of the orthonormal  polynomial $\psi_n$ with respect to $\mu$,
$$a=x_0<x_1<\cdots<x_n<x_{n+1}=b.$$
  Assume that the following conditions hold:

  \begin{enumerate}
\item[$1)$] there exists $C>0$ such that
    $$
    \|S_n f\|_{X}\le C\|f\|_{X},\quad f\in X;
    $$
    \item[$2)$] 
  there exists a nonnegative function $\vp_n$ on $[a,b]$  such that for each $r\in\Z_+$
  \begin{equation*}
    \|\vp_n^{r}P_n^{(r)}\|_Y\le B_r n^{r}\| P_n\|_Y,\quad P_n\in \mathcal{P}_n,\quad Y\in \{X,X'\},
  \end{equation*}
with some $B_r>0$.

   \item[$3)$] there exists a constant $A>0$ such that $\sum_{r=1}^\infty \frac{B_r A^r}{r!}<\infty$ and
$$
  x_{k+1}-x_k\le \frac{A}{n}\min\{\vp_n(x_k),\vp_n(x_{k+1})\},\quad k=1,\dots,n-1.
  $$
    \end{enumerate}
Then, for all $P_{n}\in \mathcal{P}_n$, we have
    \begin{equation*}
      \|P_n\|_{X}\asymp\bigg\|\sum_{k=1}^n |P_n(x_k)|\chi_{[x_{k-1},x_{k+1}]}\bigg\|_{X}.
     \end{equation*}
\end{theorem}

\begin{proof}
First we note that conditions 1), 2) and 3) coincide, after the corresponding reformulations, with conditions 3), 4) and 5) of Theorem~\ref{thOr2}. Thus, to prove Theorem~\ref{algw}, it is enough to verify conditions 1)  and 2) of Theorem~\ref{thOr2}.
By the well-known Gauss quadrature formula (see, e.g.,~\cite[p.~47]{Sego}), condition 1) of Theorem~\ref{thOr2} holds
for $N=n$, herewith $(\mu_k)_{k=1}^N$ are the Cristoffel weights. Next, property 2) in Theorem~\ref{thOr2} holds with the constants $c_1=1$ and $c_2=2$ if $\Omega_k=(x_{k-1},x_{k+1})$, $k=1,\dots,n$, see~\cite[p.~50]{Sego}. This is the well-known Chebyshev-Markov-Stieltjes separation theorem.
\end{proof}

Let us illustrate this theorem in the case of Jacobi weights.

\begin{corollary}\label{jacobi}
  Let $w_{\a,\b}(x)=(1-x)^\a(1+x)^\b$, $\a,\b>-1$ and let
$x_1,x_2,\dots,x_{n}$ be the zeros of $n$-th orthogonal Jacobi polynomial corresponding to the weight $w_{\a,\b}$. If $1<p<\infty$,  $|(\a+1)(1/2-1/p)|<\min(1/4,(\a+1)/2)$ and $|(\b+1)(1/2-1/p)|<\min(1/4,(\b+1)/2)$, then, for all $P_n\in \mathcal{P}_n$,
\begin{equation}\label{E2J}
  \int_{-1}^1|P_n(x)|^pw_{\a,\b}(x)dx\asymp\sum_{j=1}^n|P_n(x_j)|^p \mu_j,
\end{equation}
where
$$
\mu_j=\int_{x_{j-1}}^{x_{j+1}}w_{\a,\b}(x)dx,\quad x_0=-1,\,\,x_{n+1}=1,\quad j=1,\dots,n.
$$
\end{corollary}

\begin{proof}
We apply Theorem~\ref{algw} in the case of the Jacobi weight $w_{\a,\b}$ on $[a,b]=[-1,1]$, $\vp_n(x)=\sqrt{1-x^2}+\frac1n$,
and $\|f\|_X=\(\int_{-1}^1|f(x)|^pw_{\a,\b}(x)dx\)^{1/p}$. Then, the fulfillment of condition 1) follows from~\cite{Muc69}. For the corresponding Bernstein inequality in condition 2), see~\cite{MT00}. Finally, we have (see~\cite[Lemma~2.3]{Shi07}) that for all $\a,\b>-1$,
    $$
    |x_{k+1}-x_k|\asymp \frac{\vp_n(x)}n,\quad x_k\le x\le x_{k+1},
    $$
which implies condition~3) of Theorem~\ref{algw}.
\end{proof}

See also~\cite{L17} for necessary and sufficient conditions on the parameters $\a$ and $\b$ for \eqref{E2J} to hold.

\section{Examples of MZ type inequalities for other function classes}

\subsection{Band-limited functions}\label{efet} 
In what follows, the class of band-limited functions  $\mathcal{B}_X^\s$, $\s>0$, in the space $X$, is given by
%
$$
\mathcal{B}_X^\s=\left\{\varphi  \in X\cap L_1(\R^d)\,:\,\supp\;\widehat{\varphi} \subset Q(0, \s)\right\},
$$
where
$$
\widehat{g}(x) = \int_{\R^d} g(y) e^{-i(x,y)} dy.
$$

Similarly as in the periodic case, the following proprosition is a direct corollary from Theorem~\ref{thmmwe}.

\begin{proposition}\label{corn1E}

Let $X=X(\R^d,\mu)$ be a Banach function lattice such that
\begin{equation}\label{BerBLg}
  \left\|\frac{\partial}{\partial x_i}g_\s\right\|_X\le B\s\|g_\s\|_X,\quad i=1,\dots,d,\quad g_\s\in \mathcal{B}_p^\s,
\end{equation}
for some constant $B>0$. Then, for any $0<A<\frac{\log 2}{dB}$ and any $g_\s\in \mathcal{B}_p^\s$, we have
  \begin{equation*}
    \frac{1}{e^{dAB}}\Big\|\max_{t\in Q(x,A/n)}|g_\s(t)|\Big\|_X\le \|g_\s\|_X\le \frac{1}{2-e^{dAB}}\Big\|\min_{t\in Q(x,A/n)}|g_\s(t)|\Big\|_X.
  \end{equation*}
%
\end{proposition}

Note that in the case of the translation invariant lattice $X$, inequality~\eqref{BerBLg} holds with $B=1$. This can be obtained using the standard machinery involving the M. Riesz interpolation formula for band-limited functions; see, e.g., ~\cite[Ch.~XVI, \S~7]{Z}.


Proposition~\ref{corn1E} and Theorem~\ref{corthmmwe} with $ N=\infty$
yield  the following Plancherel--P\'olya-type inequality in an arbitrary translation invariant Banach lattice.

\begin{theorem}\label{PP}
  Let $X=X(\R^d,\mu)$ be a  Banach lattice and let $\e>0$. Assume that~\eqref{BerBLg} holds with the constant $B>0$.
  Then, for all $g_\s\in \mathcal{B}_X^\s$ and $W\ge (\frac{dB}{\log 2}+\e)\s$,
\begin{equation}\label{PPe}
  \|g_\s\|_X\asymp \bigg\|\sum_{k\in \Z^d} |g_\s(x_k)|\chi_{\Omega_k}\bigg\|_{X},
\end{equation}
  where $x_k$ is an arbitrary point in $\Omega_k=[\frac{k_1-1}{W},\frac{k_1}{W}]\times\dots\times [\frac{k_d-1}{W},\frac{k_d}{W}]$, $k\in \Z^d$.
\end{theorem}

In the case of quasi-Banach lattice, Theorem~\ref{corthmmwe} gives the following result.

\begin{theorem}\label{qBg}
  Let $X=X(\R^d,\mu;q)$, $0<q\le 1$, be a quasi-Banach lattice.   Assume that for all $\a\in\Z_+^d$, there exist $B_\a>0$ and $A>0$ such that
  \begin{equation*}
    \|D^\a g_\s\|_X\le B_\a \s^{|\a|_1}\| g_\s\|_X,\quad g_\s\in \mathcal{B}_p^\s,
  \end{equation*}
and $\sum_{\a\in\Z_+^d,\, \a\neq 0}\(\frac{A^{|\a|_1}B_\a}{\a!}\)^{q}<1$. Then, for all $ g_\s\in \mathcal{B}_p^\s$ and  $W\ge \frac{\s}{A}$,
\begin{equation}\label{PPegg}
  \|g_\s\|_X\asymp \bigg\|\sum_{k\in \Z^d} |g_\s(x_k)|\chi_{\Omega_k}\bigg\|_{X},
\end{equation}
  where $x_k$ is an arbitrary point in $\Omega_k=[\frac{k_1-1}{W},\frac{k_1}{W}]\times\dots\times [\frac{k_d-1}{W},\frac{k_d}{W}]$, $k\in \Z^d$.
\end{theorem}


\begin{example}
Equivalences~\eqref{PPe} and~\eqref{PPegg} hold in the following cases:
    \begin{itemize}
        \item[$(i)$] $X=L_p(\R^d)$, $0< p\le \infty$, see, e.g.,~\cite{PP37}, \cite{Nik51}, \cite{GaM}, \cite{P07}, \cite[1.4.4]{ST86}; 

        \item[$(ii)$]  $X=L_\Phi(\R)$ is the Orlicz space, where $\Phi$ satisfies conditions 1) and 2) of Definition~\ref{defO}, see~\cite{MusT89}, \cite{Bo52};

        \item[$(iii)$] $X=L_{\overline{p}}(\R^d)$ is a mixed $L_p$-space, see~\cite{TW15};

        \item[$(iv)$] $X$ is a Lorentz or  Morrey space on $\R^d$.
    \end{itemize}
\end{example}

\subsection{Splines}\label{splines-}
Denote by $\mathcal{S}_{r,n}$ the set of all periodic spline functions of
degree $r-1$, $r\ge 2$, with the knots $t_j=j/n$, $j=0,\ldots,n,$ i.e.,
$S\in \mathcal{S}_{r,n}$ if $S\in C^{r-2}[0,1]$ and $S$ is some
algebraic polynomial of degree $r-1$ in each interval
$(t_{j-1},\,t_j)$, $j=1,\ldots,n$. We also assume that $S$ is a 1-periodic function, i.e., $S(t+1)=S(t)$ for $t\in \R$.

Recall (see, e.g.,~\cite[p.~136 and p.~103]{DL}) that
\begin{equation}\label{spM}
   \| S_n'\|_{L_\infty[0,1]}\le 2r^2 n\|S_n\|_{L_\infty[0,1]},\quad S_n\in \mathcal{S}_{r,n}.
\end{equation}
Using~\eqref{spM} and repeating the proof of Proportion~\ref{pMarkX} (for a translation invariant space $X$  we can take $Q=[0,1]$, $c=0$ and avoid the use of characteristic functions), we arrive at the following result.

\begin{proposition}\label{sMarkX}
Let $X$ be a translation invariant Banach space of functions on $[0,1]$. Then, for each $S_n\in \mathcal{S}_{r,n}$,  there holds
\begin{equation*}
  \|S_n'\|_{X}\le 2r^2n\|S_n\|_{X}.
\end{equation*}
\end{proposition}

Further, applying Theorem~\ref{corMZinv} and Proposition~\ref{sMarkX}, we obtain the following theorem.

\begin{theorem}\label{mzMarkS}
  Let $X$ be a translation invariant Banach lattice on $[0,1]$ and let $\e>0$. Then, for each $S_n\in \mathcal{S}_{r,n}$ and $N\ge (2r^2+\e)n$,  we have
\begin{equation*}
  \|S_n\|_X\asymp \bigg\|\sum_{k=1}^{N}|S_n(x_k)|\chi_{[\frac{k-1}{N}, \frac{k}{N}]}\bigg\|_{X},
\end{equation*}
where $x_k$ is an arbitrary point in $[\frac{k-1}{N}, \frac{k}{N}]$, $k=1,\dots,N$.
\end{theorem}
%
%
See also~\cite{PQ91} for some results in the case $X=L_p[0,1]$, $1\le p\le \infty$.

\subsection{Exponential sums and M\"untz polynomials}\label{secME} Let $\Lambda_n=\{\l_0<\l_1<\dots<\l_n\}$ be a set of real numbers. Denote
$$
E(\Lambda_n)={\rm span}\,\{e^{\l_0 t}, e^{\l_1 t},\dots,e^{\l_n t}\}.
$$
Using the results  in~\cite{E07}, we know that for any finite interval $(a,b)\subset \R$ there exists a constant $c_1=c_1(a,b)>0$ such that
\begin{equation}\label{es1}
  \|P_n'\|_{L_\infty[a,b]}\le c_1 \g(\Lambda_n)\|P_n\|_{L_\infty[a,b]},\quad P_n\in E(\Lambda_n),
\end{equation}
where
$$
\g(\Lambda_n):=n^2+\sum_{j=0}^n \l_j.
$$
Repeating the proof of Proposition~\ref{pMarkX}, we obtain with the help of~\eqref{es1} the following result.

\begin{proposition}\label{pMarkXE}
    Let $X$ be a rearrangement invariant Banach space of functions on $[a,b]\subset \R$ and let $r\in \N$. Then, there exists a constant $c_2=c_2(a,b)>0$ such that
\begin{equation}\label{es1X}
  \|P_n^{(r)}\|_{X}\le (c_2 \g(\Lambda_n))^r\|P_n\|_{X},\quad P_n\in E(\Lambda_n).
\end{equation}
\end{proposition}

Thus, by the standard arguments described above, employing  Proposition~\ref{pMarkXE} and Theorem~\ref{corthmmwe}, we arrive at the following MZ inequality.
\begin{theorem}\label{mzMarkE}
  Let $X$ be a rearrangement invariant Banach space on $[a,b]$ and let $\e>0$. Then, for each $P_n\in E(\Lambda_n)$ and $N\ge (\frac{c_2}{\log 2}+\e)\g(\Lambda_n)$,  we have
\begin{equation*}\label{E1}
  \|P_n\|_X\asymp \bigg\|\sum_{k=1}^{N}|P_n(x_k)|\chi_{\Omega_k}\bigg\|_{X},
\end{equation*}
where $x_k$ is an arbitrary point in $\Omega_k=[a+\frac{(b-a)(k-1)}{N}, a+\frac{(b-a)k}{N}]$, $k=1,\dots,N$.
\end{theorem}

Note that Theorem~\ref{mzMarkE} generalizes~\cite[Proposition~1]{Kr22} and provides the better estimate for $N$ with respect to order in the case $X=L_p[a,b]$ for $1\le p\le 2$.

Under some additional assumptions on the space $X=X[a,b]$, similar results hold also for M\"untz polynomials belonging to
$$
 M(\Lambda_n)={\rm span}\,\{x^{\l_0}, x^{\l_1},\dots,x^{\l_n}\}.
$$
 Let $0<a<b<\infty$. Assume that there exists a rearrangement invariant  Banach lattice $\widetilde{X}$ on $[\log a,\log b]$ such that
\begin{equation}\label{XXee}
  \|f\|_X\asymp \|f(e^t)\|_{\widetilde{X}},\quad f\in X.
\end{equation}
For example this holds if $X$ is an appropriated Orlicz space. Then, using the substitution $x=e^t$, we deduce that~\eqref{es1X} implies the existence of a constant $c_3=c_3(a,b,X)$ such that
\begin{equation*}
  \|P_n'\|_{X}\le c_3\g(\Lambda_n)\|P_n\|_{X},\quad P_n\in M(\Lambda_n).
\end{equation*}
Further, by induction arguments, we obtain for each $r\in \N$,
\begin{equation}\label{ps1Xr}
  \|P_n^{(r)}\|_{X}\le r! (c_3\g(\Lambda_n))^r\|P_n\|_{X},\quad P_n\in M(\Lambda_n).
\end{equation}

Applying Theorem~\ref{corthmmwe} together with~\eqref{ps1Xr}, we get the following analogue of Theorem~\ref{mzMarkE} for M\"untz polynomials.
\begin{theorem}\label{mzMarkEE}
  Let $X$ be a rearrangement invariant Banach space on $[a,b]$, $0<a<b<\infty$, satisfying~\eqref{XXee}, and let $\e>0$. Then, for each $P_n\in M(\Lambda_n)$ and $N\ge ({2c_3}+\e)\g(\Lambda_n)$,  we have
\begin{equation*}\label{E1}
  \|P_n\|_X\asymp \bigg\|\sum_{k=1}^{N}|P_n(x_k)|\chi_{\Omega_k}\bigg\|_{X},
\end{equation*}
where $x_k$ is an arbitrary point in $\Omega_k=[a+\frac{(b-a)(k-1)}{N}, a+\frac{(b-a)k}{N}]$, $k=1,\dots,N$.
\end{theorem}

%

\section{Applications}

\subsection{Approximation by sampling operators}

Our aim here is to demonstrate how
MZ inequalities can be applied
  to obtain error estimates for sampling operators in an arbitrary translation invariant function lattice $X$ on $\T$.
We illustrate this by considering the simple model case of
  the classical Lagrange interpolation polynomial $\mathcal{L}_n f$ with respect to the points $(t_k)_{k=0}^{2n}$, $t_k=\frac{2\pi k}{2n+1}$,
$$
\mathcal{L}_nf(x)=\frac{1}{2n+1}\sum_{k=0}^{2n}f(t_k)\frac{\sin(n+\tfrac12)(x-t_k)}{\sin\frac{x-t_k}{2}}.
$$
To estimate  $\|f-\L_nf\|_{X}$, we use
the error of the best (one-sided) approximation 
 as well as  the classical and averaged moduli of smoothness,
see Section~\ref{A2} for
the definitions.




\begin{theorem}\label{thL1}
Let $X$ be a translation invariant Banach lattice on $\T$ such that
$$\|S_n g\|_X\le C\|g\|_X,\quad g\in X,$$
where $S_n$ is the Fourier partial sum operator,
and let $f\in C(\T)$.
Then
\begin{equation}\label{l1}
  \|f-\L_nf\|_{X}\le C \widetilde{E}_n(f)_{X},
\end{equation}
where the constant $C$ depends only on $X$.
\end{theorem}

\begin{proof}
Let $q_n, Q_n\in \mathcal{T}_n^1$ be such that $q_n(x)\le f(x)\le Q_n(x)$
and $\|q_n-Q_n\|_{X}=\widetilde{E}_n(f)_{X}$.
The proof is based on inequality~\eqref{EEE1t} as follows:
\begin{equation*}
  \begin{split}
\|f-\L_nf\|_{X}&\le \|f-q_n\|_{X}+\|\L_nf-q_n\|_{X}\\
&\le\|Q_n-q_n\|_{X}+C_1\bigg\|\sum^{2n}_{k=0}|f(t_k)-q_n(t_k)|\chi_{[t_k,t_{k+1}]}\bigg\|_{X}\\
&\le \widetilde{E}_n(f)_{X} +C_1\bigg\|\sum^{2n}_{k=0}|Q_n(t_k)-q_n(t_k)|\chi_{[t_k,t_{k+1}]}\bigg\|_{X}\\
&\le C\widetilde{E}_n(f)_{X}.
   \end{split}
\end{equation*}
\end{proof}

Combining Theorem~\ref{thL1} and Lemma~\ref{appX}, we arrive at analogues of~\eqref{l1} for different measures of smoothness.

\begin{corollary}\label{cor1}  Under the conditions of Theorem~\ref{thL1}, we have, for each $r\in \N$,
\begin{equation*}
  \|f-\L_nf\|_{X}\le C \tau_r\(f, \frac1n\)_{X},
\end{equation*}
where the constant $C$ depends only on $X$ and $r$. If, additionally,  $f$ is such that $f^{(s-1)}\in AC(\T)$ and $f^{(s)}\in X$ for some $s\in \N$, then
\begin{equation*}
  \|f-\L_nf\|_{X}\le \frac{C}{n^s} E_n(f^{(s)})_{X}  \le \frac{C}{n^s} \omega_r\(f^{(s)}, \frac1n\)_{X},
\end{equation*}
where the constant $C$ depends only on $X$, $r$, and $s$.
\end{corollary}

Other interesting  applications  using MZ inequalities in this direction  can be found in~\cite{KLP18}, \cite{KL23}, \cite{KP21}, \cite{Gr20}, \cite{H89}, \cite{Pe83}.

%

\subsection{Nikolskii inequalities}

We start this section with  an analogue of the Nikolskii inequality for various Banach spaces.
\begin{theorem}\label{thNik}
  Let $X$ be a translation invariant Banach lattice on $\T^d$ and let $\g>0$, $n\in\N$. Then, for any integer $N\ge \g n$, we have
  \begin{equation}\label{tnik1}
     \|T_n\|_{L_\infty(\T^d)}\le \frac{C}{\|\chi_{(0,\frac{2\pi}{N})^d}\|_X}\|T_n\|_X, \quad T_n\in \mathcal{T}_n^d,
  \end{equation}
  where $C=
  {e^{{\frac{2\pi d}{\g}}}}$. Moreover, if $0<p<q<\infty$ and a Banach lattice $Y$ is such that
  \begin{equation}\label{tnik2}
     \|f\|_Y\le K\big\||f|^{\frac qp}\big\|_X^{\frac pq},\quad f\in C(\T^d),
  \end{equation}
  for some constant $K>1$. Then
  \begin{equation}\label{tnik3}
    \|T_n\|_Y \le K\(\frac{C}{\|\chi_{(0,\frac{2\pi}{N})^d}\|_X}\)^{1-\frac pq}\|T_n\|_X, \quad T_n\in \mathcal{T}_n^d.
  \end{equation}
\end{theorem}

\begin{proof}
 The proof is based on Theorem~\ref{corMZinv}(B). We take
 $$
 \Omega_k=[\tfrac{2\pi(k_1-1)}{N}, \tfrac{2\pi k_1}{N}]\times\cdots\times [\tfrac{2\pi(k_d-1)}{N}, \tfrac{2\pi k_d}{N}],\quad k\in [1,N]^d.
 $$
 Then~\eqref{mzinv0} holds with $c_d=1$ and~\eqref{mzinv0+} is fulfilled provided  $\frac{2\pi}{N}\le \frac An$. We stress that no restrictions on  $A>0$ is required. Thus,
  inequality~\eqref{mzinv1} (see also Proposition~\ref{corn1}) gives
  \begin{equation}\label{tnik4x}
    e^{dA}\|T_n\|_X\ge \bigg\|\sum_{k_1=1}^{N}\dots\sum_{k_d=1}^{N} \max_{t\in \Omega_k}|T_n(t)|\chi_{\Omega_k}\bigg\|_{X}.
  \end{equation}
  Without loss of generality, we may assume that $\|T_n\|_{L_\infty(\T^d)}=\max_{t\in \Omega_\textbf{1}}|T_n(t)|$, where $\textbf{1}=(1,\dots,1)$. Thus, by the homogeneity of the norm, we derive from~\eqref{tnik4x} that
    \begin{equation*}
   e^{dA}\|T_n\|_X\ge \max_{t\in \Omega_\textbf{1}}|T_n(t)| \|\chi_{\Omega_\textbf{1}}\|_{X}\ge \|T_n\|_{L_\infty(\T^d)}\|\chi_{(0,\frac{2\pi}{N})^d}\|_{X},
  \end{equation*}
   which implies~\eqref{tnik1} with $\g=\frac{2\pi}{A}$.

  Now we prove~\eqref{tnik3}. Denote $\|\cdot\|_{X_1}=\||\cdot|^\frac qp\|_X^\frac pq$. Then, from~\eqref{tnik2} and~\eqref{tnik1}, we derive
  \begin{equation*}
    \begin{split}
\|T_n\|_Y&\le K\||T_n|^\frac qp\|_X^\frac pq=K\||T_n|\|_{X_1}=K\||T_n|^{1-\frac pq}|T_n|^\frac pq\|_{X_1}\\
&\le K \|T_n\|_{L_\infty(\T^d)}^{1-\frac pq}\||T_n|^\frac pq\|_{X_1}\le K\(\frac{C}{\|\chi_{(0,\frac{2\pi}{N})^d}\|_X}\|T_n\|_{X}\)^{1-\frac pq}\|T_n\|_{X}^\frac pq\\
&=K\(\frac{C}{\|\chi_{(0,\frac{2\pi}{N})^d}\|_X}\)^{1-\frac pq}\|T_n\|_{X}.
    \end{split}
  \end{equation*}
\end{proof}

\begin{remark}
  $(i)$ If $X$ is a translation invariant quasi-Banach  lattice on $\T^d$ such that the Bernstein inequality~\eqref{berweQu} holds, then~\eqref{tnik1} and~\eqref{tnik3} are valid for all $0<p<q<\infty$ and the space $Y$ satisfying~\eqref{tnik2}. In this case, the constant $C$ in~\eqref{tnik1} depends on $\g$, $d$, and $C_{\a,X}$, $\a\in \Z_+^d\setminus \{0\}$.

  $(ii)$  If $X=L_p(\T^d)$, $Y=L_q(\T^d)$, $0<p<q<\infty$, we arrive at the classical Nikolskii inequality
  $$
  \|T_n\|_{L_q(\T^d)}\le Cn^{d(\frac1p-\frac1q)}\|T_n\|_{L_p(\T^d)}.
  $$
See also \cite{ber} for  Nikolskii inequalities in symmetric spaces.

  $(iii)$  Inequalities~\eqref{tnik1} and~\eqref{tnik3} can be easily extended to the case of algebraic polynomials, splines, and entire functions of exponential type. We omit the formulations of the corresponding statements.
\end{remark}


In the case of Orlicz spaces, we arrive at the following result.

\begin{corollary}\label{corNik}
  Let $\Phi$ satisfy conditions 1) and 2) of Definition~\ref{defO}. Let also $\g>0$ and $n\in \N$. Then, for any $T_n\in \mathcal{T}_n^d$ and $N\ge \g n$, we have
  \begin{equation}\label{nik1or}
     \|T_n\|_{L_\infty(\T^d)}\le C\Phi^{-1}\(\frac{N}{2\pi}\)\|T_n\|_\Phi,
  \end{equation}
  where $C= {e^{{\frac{2\pi d}{\g}}}}$.  Moreover, if $0<p<q<\infty$ and $\Psi$ is convex such that
  \begin{equation*}
    0\le \Psi(t)\le K \Phi(t^\frac qp),\quad t\in \R,
  \end{equation*}
  for some constant $K>1$. Then
  \begin{equation}\label{nik3}
    \|T_n\|_\Psi \le K\(C\Phi^{-1}\(\frac{N}{2\pi}\)\)^{1-\frac pq}\|T_n\|_\Phi.
  \end{equation}
\end{corollary}

\begin{proof}
  Inequality~\eqref{nik1or} directly follows from~\eqref{tnik1} using the fact that
  $$
   \|\chi_{(0,\frac{2\pi}{N})^d}\|_\Phi=\inf\left\{\l\,:\, \int_{(0,{\frac{2\pi}{N}})^d} \Phi\(\frac1{\l}\)dx\le 1  \right\}=\frac1{\Phi^{-1}(\frac{N}{2\pi})}.
  $$

  To prove~\eqref{nik3}, it is enough to show that $\||T_n|^\frac pq\|_{\Phi_1}\le \|T_n\|_{\Phi}^\frac pq$ with $\Phi_1(t)=\Phi(t^\frac qp)$. Indeed,
  \begin{equation*}
    \begin{split}
       \||T_n|^\frac pq\|_{\Phi_1}&=\inf\left\{\l\,:\, \int_{\T^d} \Phi\(\frac{T_n(x)}{\l^{\frac qp}}\)dx\le 1  \right\}\\
       &=\inf\left\{\l^{\frac pq}\,:\, \int_{\T^d} \Phi\(\frac{T_n(x)}{\l}\)dx\le 1  \right\}=\|T_n\|_{\Phi}^\frac pq.
     \end{split}
  \end{equation*}
\end{proof}

\begin{corollary}\label{corant}
Let $X$ be a translation invariant Banach lattice on $\T^d$ such that
 $\|S_n f\|_X\le M\|f\|_X$ for all $f\in X$, where $S_n$ is the $n$-th Fourier partial sum operator. Let also $\g>0$ and $n\in \N$. Then, for any $f\in X$ and $N\ge \g n$, we have
  \begin{equation}\label{ant1-}
    \|S_nf\|_{L_\infty(\T^d)} \le \frac{C}{\|\chi_{(0,\frac{2\pi}{N})^d}\|_X} \|f\|_X,
  \end{equation}
where $C=M {e^{{\frac{2\pi d}{\g}}}}$.
\end{corollary}

\begin{proof}
  The proof of~\eqref{ant1-} follows from~\eqref{tnik1} applied to $T_n=S_n f$.
\end{proof}

\begin{remark}
  Corollary~\ref{corant} gives a simple proof of~\cite[Theorem~2]{AL21} under some natural additional assumptions, namely, assuming that $\g>0$ and
$\Phi$ satisfies conditions 1) and 2) of Definition~\ref{defO} and, additionally, $\Phi, \Phi^{*} \in \Delta_2$. Then there exists a constant $C>0$ such that for any function $f\in L_\Phi$ and $N\ge \g n$, we have
  \begin{equation*}
    \|S_nf\|_{L_\infty(\T^d)} \le C\Phi^{-1}\(\frac{N}{2\pi}\) \|f\|_{\Phi}.
  \end{equation*}
    This inequality follows directly from~\eqref{ant1-} and~\eqref{sn}, taking into account that the space $L_{\Phi^*}$ is reflex if and only if  $\Phi^*$ and $\Phi$ satisfy $\Delta_2$-condition, see also the proof of Theorem~\ref{th1}.
\end{remark}

\subsection{Inequalities for best approximations and moduli of smoothness}

The next results are analogues of the Ulyanov inequality in general Banach spaces, cf.~\cite{diti}, \cite{KT21}.

\begin{theorem}\label{thu}
    Let $X$ be a translation invariant Banach lattice on $\T$.  Then, for any $f\in X$, we have
  \begin{equation*}
   E_n(f)_\infty\le C\(\s_{2n}E_n(f)_X+\sum_{\nu=n+1}^\infty \frac{\s_{4\nu}}{\nu}E_\nu(f)_X\),\quad n\in\N,
  \end{equation*}
  where
  $$
  \s_\nu=\frac1{\|\chi_{(0,\frac{2\pi}{\nu})}\|_X}.
  $$
  Moreover, if $0<p<q<\infty$ and the space $Y$ is such that~\eqref{tnik2} is fulfilled, then
    \begin{equation}\label{u2}
   E_n(f)_Y\le C \(\s_{2n}^{1-\frac pq}E_n(f)_X+\sum_{\nu=n+1}^\infty \frac{\s_{4\nu}^{1-\frac pq}}{\nu}E_\nu(f)_X\),
  \end{equation}
  where $C$ is a positive constant independent of $f$ and $n$.
\end{theorem}

\begin{proof}
We will prove only~\eqref{u2}. Let $T_n\in \mathcal{T}_n$ be such that $\|f-T_n\|_X=E_n(f)_X$. We can write formally that
\begin{equation*}
  f-T_n=\sum_{k=0}^\infty (T_{2^{k+1}n}-T_{2^kn}),
\end{equation*}
which implies
\begin{equation*}
  \|f-T_n\|_X\le\sum_{k=0}^\infty \|T_{2^{k+1}n}-T_{2^kn}\|_X.
\end{equation*}
Next, applying~\eqref{tnik3} with $N=n$ and $\g=1$, we have
\begin{equation*}
  \begin{split}
     \|T_{2^{k+1}n}-T_{2^kn}\|_Y\le C\s_{2^{k+1}n}^{1-\frac pq}\|T_{2^{k+1}n}-T_{2^kn}\|_X\le 2C\s_{2^{k+1}n}^{1-\frac pq} E_{2^kn}(f)_X.
   \end{split}
\end{equation*}
Thus,
\begin{equation*}
\begin{split}
  \|f-T_n\|_Y&\le 2C\sum_{k=0}^\infty \s_{2^{k+1}n}^{1-\frac pq} E_{2^kn}(f)_X\\
  &=2C\(\s_{2n}^{1-\frac pq} E_{n}(f)_X+\sum_{k=0}^\infty \s_{2^{k+2}n}^{1-\frac pq}E_{2^{k+1}n}(f)_X\).
\end{split}
\end{equation*}
It remains only to take into account that
\begin{equation*}
  \begin{split}
     \sum_{\nu=n+1}^\infty \frac{\s_{4\nu}^{1-\frac pq}}{\nu}E_\nu(f)_X&=\sum_{k=0}^\infty \sum_{\nu=2^kn+1}^{2^{k+1}n} \frac{\s_{4\nu}^{1-\frac pq}}{\nu}E_\nu(f)_X\\
     &\ge \sum_{k=0}^\infty \frac{\s_{2^{k+2}n}^{1-\frac pq}}{2^{k+1}n}E_{2^{k+1}n}(f)_X 2^kn=\frac12 \sum_{k=0}^\infty \s_{2^{k+2}n}^{1-\frac pq}E_{2^{k+1}n}(f)_X.
   \end{split}
\end{equation*}

\end{proof}

\begin{theorem}\label{coru}
    Let $X$ be a translation invariant Banach lattice on $\T$ and let $r\in \N$.  Then, for any $f\in X$ and $h>0$, we have
  \begin{equation*}
  \omega_r(f,h)_{L_\infty}\le C\int_0^h \omega_r(f,t)_X \s\(\frac1t\)\frac{dt}{t},
  \qquad \s(u)=\frac1{\|\chi_{(0,\frac{2\pi}{u})}\|_X}.
  \end{equation*}
  Moreover, if $0<p<q<\infty$ and a translation invariant Banach lattice $Y$ is such that~\eqref{tnik2} is fulfilled, then
    \begin{equation}\label{cu2}
   \omega_r(f,h)_Y\le C\int_0^h \omega_r(f,t)_X \(\s\(\frac1t\)\)^{1-\frac pq}\frac{dt}{t},
  \end{equation}
  where $C$ is a positive constant independent of $f$ and $n$.
\end{theorem}

\begin{proof}
As in the proof of the previous result, we consider only the second inequality~\eqref{cu2}. The proof is standard and based on Corollary~\ref{corNik}, Theorem~\ref{thu}, and
the following Stechkin-Nikolskii-Boas inequality
  \begin{equation}\label{ns}
    h^r\|T_n^{(r)}\|_X\asymp \|\D_h^r T_n\|_X, \quad 0<h<\frac{\pi}{n},\quad T_n\in \mathcal{T}_n,
  \end{equation}
  where the constants in this equivalence depend only on $r$. The latter equivalence can be verified repeating
  the proof of Theorem~3.1 in~\cite{KT21} and using Lemma~\ref{lecon}.

Let $T_n\in \mathcal{T}_n$ be such that $\|f-T_n\|_X=E_n(f)_X$. We have
\begin{equation}\label{thue1}
  \begin{split}
    \w_r(f,n^{-1})_Y\le 2^r\|f-T_n\|_Y+\w_r(T_n,n^{-1})_Y.
  \end{split}
\end{equation}
By \eqref{ns}, \eqref{tnik3}, and the Jackson inequality~\eqref{jack}, we have
\begin{equation}\label{thue2}
  \begin{split}
    \w_r(T_n,n^{-1})_Y&\le Cn^{-r} \|T_n^{(r)}\|_Y\le C(\s(n))^{1-\frac pq}n^{-r}\|T_n^{(r)}\|_X\\
    &\le C(\s(n))^{1-\frac pq}\w_r(T_n,n^{-1})_X\\
    &\le C(\s(n))^{1-\frac pq}\(\|f-T_n\|_X+\w_r(f,n^{-1})_X\)\\
    &\le C(\s(n))^{1-\frac pq}\w_r(f,n^{-1})_X.
  \end{split}
\end{equation}
Now, combining~\eqref{thue1}, \eqref{thue2}, \eqref{u2}, and one more time~\eqref{jack}, we obtain
\begin{equation*}
  \w_r(f,n^{-1})_Y\le C\((\s(2n))^{1-\frac pq}\w_r(f,n^{-1})_X+\sum_{\nu=n+1}^\infty \frac{(\s(4\nu))^{1-\frac pq}}{\nu}\w_r(f,\nu^{-1})_X\),
\end{equation*}
which by monotonicity implies~\eqref{cu2}.
\end{proof}

In the case of Orlicz spaces, we arrive at the following result.

\begin{corollary}
      Let $\Phi$ satisfy conditions 1) and 2) of Definition~\ref{defO}. Then, for any $f\in L_\Phi$, we have
  \begin{equation*}
   E_n(f)_{L_\infty}\le C\(\Phi^{-1}(2n)E_n(f)_\Phi+\sum_{\nu=n+1}^\infty \frac{\Phi^{-1}(4\nu)}{\nu}E_\nu(f)_\Phi\),\quad n\in\N,
  \end{equation*}
  and
    \begin{equation*}
  \omega_r(f,h)_{L_\infty}\le C\int_0^h \omega_r(f,t)_\Phi {\Phi^{-1}\(\frac1t\)}\frac{dt}{t},\quad h>0.
  \end{equation*}
  Moreover, if $0<p<q<\infty$ and $\Psi$ is convex and such that~\eqref{tnik2} is fulfilled, then
    \begin{equation*}
   E_n(f)_\Psi\le C \((\Phi^{-1}(2n))^{1-\frac pq}E_n(f)_\Phi+\sum_{\nu=n+1}^\infty \frac{(\Phi^{-1}(4\nu))^{1-\frac pq}}{\nu}E_\nu(f)_\Phi\),\quad n\in \N,
  \end{equation*}
  and
  \begin{equation*}
   \omega_r(f,h)_\Psi\le C\int_0^h \omega_r(f,t)_\Phi \({\Phi^{-1}\(\frac1t\)}\)^{1-\frac pq}\frac{dt}{t},\quad h>0,
  \end{equation*}
  where $C$ is a positive constant independent of $f$, $n$, and $h>0$.
\end{corollary}

\subsection{Embeddings of  Besov spaces}

Define a positive monotone function $\xi$ on $(1,\infty)$ satisfying the $\Delta_2$-condition, i.e., $\xi(2t)\lesssim \xi(t)$. 
Let $X$ be a translation invariant Banach lattice on $\T$.
For $r\in \N$, $s>0$, we
define the Besov space with the generalized smoothness $\xi$ as follows
\begin{equation*}
B_{X,s}^{\,\xi(\cdot)}({\T}):=
\Big\{ f \in X({\T}): |f|_{B_{X,s}^{\,(\cdot)}({\T})}=\Big(\int_0^1 \Big({\xi(1/t)} \omega_r(f,t)_X \Big)^s\frac{dt}{t}\Big)^{\frac1s} <\infty \Big\}.
\end{equation*}
Here we assume that  $\int_0^1 t^{rs} \xi^s(1/t)\frac{dt}{t}<\infty$, which guaranties that the Besov space is nontrivial.
If $\xi(x)=x^\alpha$ and $X=L_p$, then $B_{X,s}^{\,\xi(\cdot)}({\T})$ coincides with the classical Besov space $B_{p,s}^{\,\alpha}({\T}).$
\begin{corollary}\label{besov}
 Let $s>0$, $0<p<q<\infty$ and let the Banach lattices $X$ and $Y$ be the same as in Theorem~\ref{coru}. Let also
     $\displaystyle
     \s(u)=\frac1{\|\chi_{(0,\frac{2\pi}{u})}\|_X}.$
  Then, for any $f\in X$, we have
  \begin{itemize}
    \item[{\textnormal{(i)}}] $\qquad
B_{X,s}^{\,\xi_1}(\T)\hookrightarrow B_{Y,s}^{\,\xi}(\T)\quad\mbox{with}\quad \xi_1(t)=\xi (t)
\big({\s(t)}\big)^{1-\frac pq},
$
    \item[{\textnormal{(ii)}}] $\qquad
B_{X,s}^{\,\xi_1}(\T)\hookrightarrow B_{L_\infty,s}^{\,\xi}(\T)\quad\mbox{with}\quad \xi_1(t)=\xi (t)
{\s(t)},
$
  \end{itemize}
provided that
\begin{equation}\label{vsp--}\int_1^\delta
{\xi(t)}\frac{dt}{t}\lesssim {\xi(\delta)}  \quad\mbox{for any}\quad \delta>2.
\end{equation}
\end{corollary}
\begin{proof}
We prove only (ii) since the proof of (i) is similar. Note that \eqref{cu2} is equivalent to
$$   \omega_r(f,2^{-\nu})_Y\le C\sum_{k=\nu}^\infty  \omega_r(f,2^{-k})_X \big({\sigma(2^{k})}\big)^{1-\frac pq}.
$$
Now we use the following Hardy inequality (see \cite[Lemma 2.5]{PST}): for any
 $0< s\le\infty$
and nonnegative sequences $\{a_n\}$ and $\{b_n\}$
such that
$\sum_{k=1}^n b_k\lesssim b_n$ for every $n\in \N$, we have
		\[
		\sum_{n=1}^\infty  \bigg(b_n \sum_{k=n}^\infty a_k\bigg)^s\lesssim\sum_{n=1}^\infty b_n^s a_n^s.
		\]
Thus,
\begin{eqnarray*}
|f|^s_{B_{Y,q}^{\,\xi}(\T)}
&\asymp&
\sum_{\nu=1}^\infty  \omega^s_r(f,2^{-\nu})_Y {\xi(2^{\nu})}^s
\\
&\lesssim&
\sum_{\nu=1}^\infty  {\xi(2^{\nu})}^s \Big(\sum_{k=\nu}^\infty  \omega_r(f,2^{-k})_X \big({\sigma(2^{k})}\big)^{1-\frac pq}
\Big)^s
\\
&\lesssim&
\sum_{\nu=1}^\infty  {\xi(2^{\nu})}^s \({\sigma(2^{\nu})}\)^{{s(1-\frac pq)}}   \omega_r(f,2^{-\nu})_X^s \asymp |f|^s_{B_{X,q}^{\,\xi_1}(\T)}
\end{eqnarray*}
provided that
 (\ref{vsp--}) holds.
\end{proof}


\appendix

\section{}

\subsection{Results from functional analysis}

The following lemma is well known and can be proved using the duality arguments. 

\begin{lemma}\label{lecon}
  Let $X$ be a Banach space $\Omega$  such that there exists a constant $c\ge 1$ such that
  \begin{equation*}
    \|f(\cdot+t)\|_X\le c\|f(\cdot)\|_X\quad\text{for all}\quad t\in \Omega\quad \text{and}\quad f\in X,
  \end{equation*}
  then
  \begin{equation*}
    \bigg\|\int_\Omega f(\cdot-t)g(t)dt\bigg\|_X\le c\|g\|_1\|f\|_X.
  \end{equation*}
\end{lemma}


The following result about interpolation of linear functions see in~\cite{BE}.

\begin{lemma}\label{ilf}
  Let $C(Q)$ be the set of complex valued continuous functions on the compact Hausdorff space $Q$. Let $\Phi_n$ be an $n$-dimensional linear subspace of $C(Q)$ over $\C$. Let $L\neq 0$ be a linear functional on $\Phi_n$. Then there exist points $x_1,x_2,\dots,x_{2n-1}$ in $Q$ and nonzero numbers $\l_1,\l_2,\dots,\l_{2n-1}$ such that
  $$
  L(s)=\sum_{i=1}^{2n-1} \l_i s(x_i),\quad s\in \Phi_n,
  $$
  and
  $$
  \|L\|=\sup\{|L(s)|\,:\, s\in \Phi_n,\,\, \|s\|_{L_\infty(Q)}\le 1\}=\sum_{i=1}^{2n-1}|\l_i|.
  $$
\end{lemma}

\subsection{Direct approximation inequalities}\label{A2}

Let $X$ be a translation invariant Banach lattice on $\T$. The error of best approximation of $f\in X$ and the error of best one-sided approximation by trigonometric polynomials of degree at most $n$ are given by
$$
E_n(f)_X=\inf\{\Vert f-T_n\Vert_X\,:\, T_n\in \mathcal{T}_n\}
$$
and
$$
\widetilde{E}_n(f)_X=\inf\{\Vert Q_n-q_n\Vert_X\,:\,  Q_n,q_n \in \mathcal{T}_n,\quad q_n(x)\le f(x)\le Q_n(x)\},
$$
respectively.

The $r$th modulus of smoothness of a function $f\in X$ is defined by
\begin{equation*}
    \omega_r(f,\d)_X=\sup_{0<h<\d} \Vert \D_h^r f\Vert_X,
\end{equation*}
where
$$
\D_h^r f(x)=\sum_{\nu=0}^r\binom{r}{\nu}(-1)^{\nu} f(x+(r-\nu)h),
$$
$\binom{r}{\nu}=\frac{r (r-1)\dots (r-\nu+1)}{\nu!},\quad \binom{r}{0}=1$.
The $r$th averaged modulus of smoothness ($\tau$-modulus) of  $f\in C(\T)$ is defined by
$$
\tau_r(f,\d)_X=\|\omega_r(f,\cdot,\d)\|_X,
$$
where
$$
\omega_r(f,x,\d)=\sup\left\{|\D_h^r f(t)|\,:\,t,t+rh\in [x-r\d/2,x+r\d/2]\right\}
$$
is the local modulus of smoothness of $f$. 
See~\cite{SP} for more information on the averaged moduli of smoothness. 

\begin{lemma}\label{appX}
  Let $X$ be a translation invariant Banach lattice $X$ on $\T$. Then

\begin{enumerate}
    \item[$(i)$]  for each $f\in X$ and $r\in \N$,
  \begin{equation}\label{jack}
  E_n(f)_X\le C \omega_r\(f, \frac1n\)_{X};
\end{equation}

\item[$(ii)$] for each $f\in B(\T)$ and $r\in \N$,
\begin{equation}\label{jacktau}
  \widetilde{E}_n(f)_X\le C \tau_r\(f, \frac1n\)_{X};
\end{equation}

\item[$(iii)$] for each $f$ such that $f^{(r-1)}\in AC(\T)$ and $f^{(r)}\in X$ for some $r\in \N$,
\begin{equation}\label{rrr}
  \widetilde{E}_n(f)_{X} \le \frac{C_r}{n^r} E_n(f^{(r)})_{X}.
\end{equation}
\end{enumerate}
\end{lemma}

\begin{proof}
  Inequalities \eqref{jack}, \eqref{jacktau}, and~\eqref{rrr} can be proved repeating step-by-step the proofs of the corresponding results in $L_p$-spaces and employing Lemma~\ref{lecon} instead of convolution inequalities or Minkovski's inequality in those proofs. In particular, the proof of~\eqref{jack} is based on the proof of~\cite[Theorem~2.3]{DL}, see also~\cite[Remark, P.~205]{DL}. For the proof~\eqref{jacktau} see~\cite[Theorems 8.2 and 8.3]{SP}. Concerning inequality~\eqref{rrr}, the required scheme can be found in~\cite[Theorem~8.1]{SP}. See also~\cite{Ra84} for the proof of \eqref{jack} in the case of Orlicz spaces.
\end{proof}

\subsection{Various function spaces.}\label{Afs}
In what follows $\mu$ is some positive $\s$-finite measure on $\Omega\subset \R^d$.
Below we list some important spaces and their (quasi-)norms, which we use as typical examples of $X$ in our  results:

 I) \emph{Weighted $L_p$ spaces}  $L_{p,w}(\Omega)$, $0<p\le \infty$, with
 $$
 \|f\|_{p,w}=\(\int_{\Omega}|f(x)|^p w(x)dx\)^{1/p},
 $$
 where $w$ is some positive weighted function on $\Omega$. Recall that the weight $w$ is called doubling on $\Omega$ if
$\int_{2B}w\le L\int_Bw$ for some $L>0$ and any ball $B\subset \Omega$. 

II) \emph{Mixed Lebesgue spaces} $L_{\overline{p}}(\T^d)$, $\overline{p}=(p_1,\dots,p_d)$, $0< p_i\le \infty$, $i=1,\dots,d$, with
$$
\|f\|_{L_{\overline{p}}(\T^d)}=\(   \int_{\T}\(\dots\(\int_{\T}|f(x_1,\dots,x_d)|^{p_1}dx_1\)^{p_2/p_1}\dots\)^{p_d/p_{d-1}}dx_d   \)^{1/p_d}.
$$

III) \emph{Orlicz spaces} $L_{\Phi}=L_{\Phi}(\Omega,\mu)$
with the Luxemburg norm
\begin{equation*}
  \|f\|_{\Phi}=\inf\bigg\{\l\,:\,\int_{\Omega} \Phi\(\frac{f(t)}{\l} \)d\mu(t)\le 1\bigg\}.
\end{equation*}
In the above definition $\Phi\,:\, \C\mapsto \R_+$ satisfies conditions 1) and 2) of Definition~\ref{defO} below. For such functions $\Phi$, the space $L_{\Phi}$  consisting of measurable on $\Omega$ functions $f$ with the finite norm $\|f\|_{\Phi}$
is a Banach space, see~\cite{KR61} for this and equivalent definitions of Orlicz spaces.

\begin{definition}\label{defO}
A function $\Phi\,:\, \C\to \R_+$ is an Orlicz function (or $N$-function) if the following conditions hold:
\begin{itemize}
  \item[$1)$] $\Phi$ is continuous, even and convex;
  \item[$2)$] $\Phi(x)=0$ if and only if $x=0$;
  \item[$3)$] $\underset{x\to 0}{\lim}\frac{\Phi(x)}{x}=0$ and $\underset{x\to \infty}{\lim}\frac{\Phi(x)}{x}=\infty$.
\end{itemize}
\end{definition}

Recall also that $\Phi$ satisfies the $\Delta_2$-condition, or simply $\Phi\in\Delta_2$, if there exist $u'>0$ and $K>1$ such that
$$
\Phi(2u)\le K\Phi(u)\quad\text{for all}\quad u\ge u'.
$$
The condition $\Delta_2$ plays an important role in the theory of Orlicz spaces. For example, it is well known that the Orlicz space $L_\Phi$ is reflexive if and only if $\Phi\in \Delta_2$ and $\Phi^*\in \Delta_2$, where
$$
\Phi^*(v)=\sup_{u\in \R}\{uv-\Phi(u)\}
$$
is the complementary function of $\Phi$.

%
%
%
%
%

IV) \emph{Lorentz spaces} $L_{p,q}(\Omega)$, $0<p,q<\infty$,  with the quasi-norm
\begin{equation*}
  \|f\|_{p,q}=\(\int_0^\infty \(t^{1/p}f^*(t)\)^q\frac{dt}{t}\)^{1/q},
\end{equation*}
where
$$
f^*(t)=\inf\{y\,:\, \mu(\{x\in \Omega\,:\, |f(x)|\ge y\})\le t\}.
$$

V) \emph{Variable Lebesgue spaces} $L_{p(\cdot)}$ with the norm
\begin{equation*}
  \|f\|_{p(\cdot)}=\inf\bigg\{\l\,:\,\int_{\Omega} \bigg|\frac{f(x)}{\l}\bigg|^{p(x)}d\mu(x)\le 1\bigg\},
\end{equation*}
where $p$ is a  function satisfying   $1 \le p_-\le p(x)\le p_+ < \infty$ on $\Omega$.  In the theory of variable Lebesgue spaces, the following log-H\"older condition plays an important role:
\begin{equation}\label{log}
    |p(x)-p(y)|\le \frac{C}{\ln\frac{1}{|x-y|}}\quad\text{for all}\quad x,y\in\Omega\quad\text{with}\quad |x-y|\le \frac{1}{2},
\end{equation}
where $C$ does not depend on $x$ and $y$.
See, e.g.,~\cite{KMRS16} for the main properties of these spaces.

VI) Morrey space $M_{p,q}^\l(\Omega)$, $0<p<\infty$, $0<q\le \infty$ and $0\le \l\le \frac dp$, with the quasi-norm
$$
\|f\|_{M_{p,q}^\l}=\(\int_0^\infty \(r^{-\l} \sup_{x\in \Omega} \|f\|_{L_p(Q(x,r))}\)^q\frac{dr}{r}\)^{1/q}.
$$


Note that all mentioned above spaces are examples of the so-called ball spaces.


\bigskip
\bigskip
\bigskip

\begin{center}
    \textbf{Acknowledgements}
\end{center}

This research was supported through the program "Oberwolfach Research Fellows" by the Mathematisches Forschungsinstitut Oberwolfach in 2023.
The authors would also like  to thank the Isaac Newton Institute for Mathematical Sciences, Cambridge, for support and hospitality during the programme "Discretization and recovery in high-dimensional spaces", where work on this paper was undertaken. This work was supported by EPSRC grant EP/R014604/1

                         \end{document}